\documentclass[12pt]{amsart}

\usepackage[utf8]{inputenc}
\usepackage{amssymb}
\usepackage{epsfig}
\usepackage{color}
\usepackage{esint}
\usepackage{latexsym}
\usepackage{graphicx}
\usepackage{subcaption}
\usepackage{comment}
\usepackage{fancyhdr}

\graphicspath {
    {./}
    {./Figures/}
}

\newcommand\weakto{\mathop{\rightharpoonup}}
\numberwithin{equation}{section}
\newtheorem{theorem}{Theorem}[section]
\newtheorem{definition}[theorem]{Definition}
\newtheorem{assumption}[theorem]{Assumption}

\newtheorem{example}[theorem]{Example}
\newtheorem{remark}[theorem]{Remark}

\newtheorem{proposition}[theorem]{Proposition}
\newtheorem{lemma}[theorem]{Lemma}

\newcommand\R{\mathbb{R}}

\newcommand\N{\mathbb{N}}

\newcommand\Div{\mathrm{div\,}}

\newcommand\Tr{\mathop{\mathrm{Tr}}}

\newcommand\cof{\mathop{\mathrm{cof}}}
\newcommand\eps{\varepsilon}
\newcommand\calL{\mathcal{L}}

\newcommand\calE{\mathcal{E}}

\newcommand\calD{\mathcal{D}}
\newcommand\sym{\mathrm{sym}}

\def\diag{\mathrm{diag}}

\newcommand{\calH}{\mathcal{H}}
\DeclareMathOperator{\curl}{curl}

\def\O{\mathrm{O}}
\def\SO{\mathrm{SO}}
\def\nabla{D}
\renewcommand{\epsilon}{\varepsilon}

\newcommand\dist{\operatorname{dist}}

\newcommand\Id{\operatorname{Id}}

\renewcommand\div{\mathrm{div\,}}

\newcommand{\Rnn}{\R^{n\times n}}
\newcommand{\Rnnn}{\R^{n\times n\times n}}
\newcommand{\Rnnnn}{\R^{n\times n\times n\times n}}
\newcommand{\Rnsym}{\R^{n\times n}_\sym}

\newcommand\loc{\mathrm{loc}}

\newcommand\Emb{\calE}
\newcommand\Eg{\calE_0}
\newcommand\divcurl{$\operatorname{div-curl}$}
\newcommand\Divcurl{$\operatorname{div-curl}$}
\newcommand{\Xpq}{X_{p,q}}
\newcommand{\Dloca}[1]{\calD_{\loc,#1}}
\newcommand{\DlocT}{\Dloca{T}}
\newcommand\Ext{\mathop{Ext}}

\begin{document}

\title{Data-Driven Finite Elasticity}

\author[S.~Conti, S.~M\"uller and M.~Ortiz]
{S.~Conti$^1$, S.~M\"uller$^{1,2}$ and M.~Ortiz$^{1,2,3}$}

\address
{
  $^1$ Institut f\"ur Angewandte Mathematik, Universit\"at Bonn,
  Endenicher Allee 60,
  53115 Bonn, Germany.
}
\address
{
  $^2$ Hausdorff Center for Mathematics,
  Endenicher Allee 60,
  53115 Bonn, Germany.
}

\address
{
  $^3$ Division of Engineering and Applied Science,
  California Institute of Technology,
  1200 E.~California Blvd., Pasadena, CA 91125, USA.
}

\begin{abstract}
We extend to finite elasticity the Data-Driven formulation of geometrically linear elasticity presented
in Conti, Müller, Ortiz, Arch.\ Ration.\ Mech.\ Anal.\ 229, 79-123, 2018.
The main focus of this paper concerns the formulation of a suitable framework in which the Data-Driven problem of finite elasticity is well-posed in the sense of existence of solutions. We confine attention to deformation gradients $F \in L^p(\Omega;\Rnn)$ and first Piola-Kirchhoff stresses $P \in L^q(\Omega;\Rnn)$, with $(p,q)\in(1,\infty)$ and $1/p+1/q=1$. We assume that the material behavior is described by means of a material data set containing all the states $(F,P)$ that can be attained by the material, and develop germane notions of coercivity and closedness of the material data set. Within this framework, we put forth conditions ensuring the existence of solutions. We exhibit specific examples of two- and three-dimensional material data sets that fit the present setting and are compatible with material frame indifference.
\end{abstract}

\maketitle

\section{Introduction}

In a recent publication \cite{ContiMuellerOrtiz2018-Datadriven}, we have proposed a reformulation of geometrically linear elasticity in which solutions are understood as points in stress-strain space, or {\sl phase space}, satisfying compatibility and equilibrium constraints. The reformulation was motivated by earlier work on Data-Driven methods in computational mechanics \cite{WD005}, in which the aim is to formulate solvers for boundary-value problems directly from material data, thus eschewing material modeling altogether. The Data-Driven problem, as defined in \cite{WD005} and in \cite{ContiMuellerOrtiz2018-Datadriven}, consists of minimizing the distance between a given material data set and the subspace of compatible strain fields and stress fields in equilibrium. It is not immediately clear that such problems are well-posed in the sense of existence of solutions, especially when the material data is in the form of unstructured, possibly 'noisy', point sets. These difficulties notwithstanding, in \cite{ContiMuellerOrtiz2018-Datadriven} conditions on the data set ensuring existence are set forth, and it is shown that classical solutions are recovered when the data set takes the form of a graph.
{The approach developed in  \cite{ContiMuellerOrtiz2018-Datadriven} has been used for the study of relaxation in a data-driven model with a single outlier, which was motivated by the study of porous media \cite{RoegerSchweizerDD}.} 

This latter connection shows that Data-Driven problems generalize classical problems and subsume them as special cases. The broad point of view that the problems of continuum mechanics can be written as a set of linear partial-differential equations (balance laws) and nonlinear pointwise relations between the quantities in the balance laws (constitutive relations) was propounded by Luc Tartar in the 1970s. For an early exposition of these ideas, see \cite{Tartar1979}. For further developments, see \cite{Tartar1985,Tartar1990_Hmeas} and the monograph \cite{Tartar2009Buch}. This focus on material data as the main source of epistemic uncertainty, embedded within the framework of universal or material-independent balance laws, comes back naturally to the fore in connection with the current interest in Data Science (see, for example, \cite{WD002, WD023, WD005}).

The main focus of the present paper concerns the formulation of a suitable framework in which the Data-Driven problem of finite elasticity is well-posed in the sense of existence of solutions.  The new Data-Driven finite elasticity framework is an extension of that developed in \cite{ContiMuellerOrtiz2018-Datadriven} in the geometrically linear setting. We recall that, in that case, the natural phase space of strains and stresses is $\mathcal{Z} := L^2(\Omega, \R^{n\times n}_{\rm sym}) \times L^2(\Omega, \R^{n\times n}_{\rm sym})$. In addition, conservation of angular momentum and material-frame indifference are built directly into the phase space, simply by restricting it to symmetric stresses and strains. Furthermore, the Hilbert-space structure of $\mathcal{Z}$ greatly facilitates analysis. In extending the theory to finite kinematics, much of this convenient structure is lost and needs to be generalized.

The new framework is laid out in Section~\ref{4lPRew}. We confine attention to phase spaces of the form $\Xpq(\Omega) := L^p(\Omega;\Rnn)\times L^q(\Omega;\Rnn)$, which combine deformation gradients $F \in L^p(\Omega;\Rnn)$ and first Piola-Kirchhoff stresses $P \in L^q(\Omega;\Rnn)$, with $(p,q)\in(1,\infty)$ and $1/p+1/q=1$. In addition to its usual strong and weak topologies, we endow the phase space $\Xpq(\Omega)$ with two additional topologies: A topology of \divcurl-convergence and a topology of data or $\Delta$-convergence.
{A sequence $(F_k,P_k)$ is \divcurl-convergent if it converges weakly and the sequences $\curl F_k$, $\Div P_k$ converge strongly in the corresponding negative spaces, so that one can apply the div-curl-Lemma.} 
{A sequence $(y_k,z_k)$ is $\Delta$ convergent if it converges weakly and the sequence of differences $y_k-z_k$ converges strongly.} 
{We refer to}
 Definition~\ref{defdivcurlc} and Definition~\ref{defdeltac} {for details}. These topologies play an important role in establishing conditions for existence.

The state $(F,P)\in \Xpq(\Omega)$ of the elastic body is subject to compatibility and equilibrium constraints.
{The compatibility constraint is linear and requires $F$ to be the gradient of a displacement field, 
$F=Du$,}
together with Dirichlet conditions on part of the boundary.
{The equilibrium constraint is also linear and requires the divergence of $P$ to be in equilibrium with body forces, 
$\div P+f=0$,}
together with Neumann conditions on the remainder of the boundary.
In addition, the state $(F,P)$ of the elastic body must satisfy moment equilibrium, {$FP^T=P^TF$.}
These constraints, taken together, define a constraint subset $\Emb \subseteq \Xpq(\Omega)$ of admissible states
{(we refer to Definition \ref{defineE} for details)}. Unlike the case of linearized kinematics, $\Emb$ is not an affine subspace of $\Xpq(\Omega)$ due to the nonlinearity of the moment-equilibrium constraint $FP^T=PF^T$. However, this nonlinearity can be treated {\sl via} compensated compactness, and indeed we show in Lemma~\ref{lemmacoercFP} that $\Emb$ is weakly closed in $\Xpq(\Omega)$.

In addition, as in the case of linear elasticity \cite{ContiMuellerOrtiz2018-Datadriven}, we assume that the material behavior is characterized by a material data set $\calD \subseteq \Xpq(\Omega)$ collecting all the possible states attainable by the material. An important particular case is that of local materials.
In that case, there is a local material data set $\calD_\loc \subseteq \Rnn\times\Rnn$ that contains the possible local states $(F(x),P(x))$ of the material, see Definition~\ref{deflocalmat}. In the context of finite kinematics, material-frame indifference additionally requires the local material data set $\calD_\loc$ to be invariant under the action of the full proper orthogonal group $SO(n)$. By virtue of material-frame invariance, $\calD_\loc$ consists of $SO(n)$-orbits, see Remark~\ref{remarkorbit}.

On this basis, the Data-Driven problem of finite elasticity consists of determining the state $(F,P)$ in the constraint set $\Emb$ that minimizes its deviation from a material data set $\calD$. In order to measure that deviation, we choose a convex function $V:\Rnn\to[0,\infty)$, with convex conjugate $V^*$, such that both vanish only at 0, and define the function of state
\begin{equation}\label{SP5voK}
    \psi(F,P)=\min\{ V(F-F')+V^*(P-P') \, : \, (F',P') \in \calD_\loc \}.
\end{equation}
By an appropriate choice of $V$ and $V^*$ (see Assumption~\ref{Ze6How}), $\psi$ is non-negative, vanishes on $\calD_\loc$ and grows away from it, thus providing a measure of deviation from $\calD_\loc$. The Data-Driven problem of finite elasticity is, then,
\begin{equation}\label{9lFapr}
    \inf_{(F,P)\in\Emb} \int_\Omega \psi(F(x),P(x)) \, dx ,
\end{equation}
or, equivalently,
\begin{equation}\label{6Hudof}
    \inf_{(F,P)\in\Emb,\ (F',P')\in \calD}
    \int_\Omega \Big( V(F(x)-F'(x))+V^*(P(x)-P'(x)) \Big) \, dx .
\end{equation}
In this latter form and depending on the order of minimization, the aim of the Data-Driven problem of finite elasticity is to find admissible states, i.~e., states satisfying compatibility and equilibrium, that are closest to the material data set, or, equivalently, states in the material data set that are closest to being admissible, with 'closeness' understood locally in the sense of the function $\psi(F(x),P(x))$.

We remark that the nonlinear condition of moment equilibrium renders the constraint set $\Emb$ difficult to approximate. In practice, it is often more convenient to work with the affine space $\Eg$ and build the moment balance condition into the data set, as specified in Definition \ref{defmomequil}. This alternative choice is a simple matter of convenience and does not entail {an} essential change to the framework presented here.

In the remainder of the paper, we formulate conditions under which the Data-Driven problem of
 finite elasticity (\ref{6Hudof}) has solutions. To this end, we follow the standard direct method of the Calculus of Variations. 
Our main compactness result is Theorem~\ref{theoremcompactness}, which establishes the weak, $\Delta$ and \divcurl-relative compactness of sequences of admissible and material states whose deviation from each other remains bounded. The key assumption is a property of the data set, referred to as $(p,q)$-coercivity, which we introduce in Definition~\ref{defcoercive}.

In Section~\ref{V7bisp} we elucidate the requisite lower-semicontinuity of (\ref{9lFapr}) in terms of closedness properties of the material data set.
The appropriate notion  turns out to be closedness with respect to \divcurl-convergence
(Definition~\ref{definitiondeltaclosed}). We note that, in following this approach, we depart from that of  \cite{ContiMuellerOrtiz2018-Datadriven}, which focuses instead on the closedness of $\calE \times \calD$ with respect to $\Delta$-convergence.
 Whereas the approach with $\Delta$-convergence and transversality deals with both $\calE$ and $\calD$ jointly, the present formulation in terms of \divcurl-convergence permits to phrase a large part of the discussion solely in terms of the data set $\calD$, see for example Definition \ref{defcoercive}, Definition~\ref{definitiondeltaclosed} and Definition~\ref{definitionlocdeltaclosed}. In particular, coercivity and closedness depend only on the data set and not on the external forcing and the boundary conditions. In Proposition~\ref{propdeltaclosed}, we show that \divcurl-convergence can be elucidated locally in terms of $\calD_\loc$, which greatly facilitates the analysis of specific material data sets. In Sections~\ref{Jlt9at} and \ref{dRl1EW} we present examples of \divcurl-closed material data sets generated by stress-strain functions, defining graphs in local phase space, which are compatible with material frame indifference. In all these examples, the property of polymonotonicity of the stress-strain function plays an important role and supplies sufficient conditions for \divcurl-closedness. As usual in finite elasticity, convexity is incompatible with material frame indifference.

Existence of solutions then follows from lower-semicontinuity and compactness by Tonelli's theorem \cite{Tonelli1921}. A number of alternatives arise in connection with the possible solutions of problem (\ref{6Hudof}). A Data-Driven solution consists of two fields, $(F,P)\in\Emb$ and $(F',P')\in \calD$, which may or may not coincide. We focus here on the case in which the infimum is zero, corresponding to solutions with $(F,P) = (F',P')$ that we call classical. Specifically,  Theorem~\ref{prA5lj} ensures existence of classical solutions if $\calD_\loc$ is $(p,q)$-coercive, \divcurl-closed and, in addition, the infimum of (\ref{9lFapr}) is $0$.

The general Data-Driven setting allows in principle also for the existence of generalized solutions,
with $(F,P) \neq (F',P')$. {These} solutions may arise in practice from data sets, such are commonly obtained from experiment, consisting of finite point sets. For such data sets, $\Emb$ and $\calD$ are likely to be disjoint even in cases in which a classical solution might {be} expected to exist. The Data-Driven reformulation of the problem supplies a workable notion of 'best solution' under such circumstances. Generalized solutions are not addressed in this paper.

Evidently, the infimum in (\ref{6Hudof}) may not be realized at all, in which case Data-Driven solutions, classical or generalized, fail to exist and the functional (\ref{6Hudof}) needs to be relaxed. Again, the question of relaxation is beyond the scope of this paper and is deferred to future work.

\section{General formulation}\label{4lPRew}

We consider throughout the entire paper  an elastic body occupying a set $\Omega$ in its reference configuration and assume the following.

\begin{assumption}\label{assumptiongeneral}
\begin{itemize}
\item[i)] $\Omega \subseteq \R^{n}$ is bounded, connected, open, Lip\-schitz with outer normal $\nu:\partial\Omega\to S^{n-1}$.
\item[ii)] The sets $\Gamma_D$, $\Gamma_N\subseteq\partial\Omega$ are disjoint relatively open subsets of $\partial\Omega$ with $\overline{\Gamma}_D \cap \overline{\Gamma}_N = \partial\Omega$, $\calH^{n-1}(\overline\Gamma_N \setminus \Gamma_N)=\calH^{n-1} (\overline\Gamma_D\setminus\Gamma_D)=0$, and $\Gamma_D \ne \emptyset$.
\item[iii)] We let $(p,q)\in(1,\infty)$ and $1/p+1/q=1$.
\item[iv)] The body deforms under the action of applied body forces $f\in L^q(\Omega;\R^n)$, prescribed boundary displacements $g_D \in W^{1/q,p}( \partial\Omega ; \R^{n})$ and applied boundary tractions $h_N \in W^{-1/q,q}(\partial\Omega; \R^{n})$.
\end{itemize}
\end{assumption}

We refer to Appendix~\ref{appendixtraces} and, in particular, Definition~\ref{deffractionalsobolev} for the definition of the fractional and negative Sobolev spaces.

We begin by defining suitable phase spaces, i.~e., spaces of work-conjugate deformations and stresses $(F(x),P(x))$ over $\Omega$, and endowing them with several topologies.

\begin{definition}[Phase space]\label{defxpq}
For $p,q\in(1,\infty)$, $\Omega\subseteq\R^n$ open, we define the phase space by $\Xpq(\Omega):=L^p(\Omega;\Rnn)\times L^q(\Omega;\Rnn)$ with the norm
\begin{equation}
    \|(F,P)\|_{\Xpq(\Omega)}
    :=
    \|F\|_{L^p(\Omega;\Rnn)}+\|P\|_{L^q(\Omega;\Rnn)}.
\end{equation}
\end{definition}
In the following we shall for brevity often drop the target space from the norms, writing
for example $    \|F\|_{L^p(\Omega)}$ instead of $\|F\|_{L^p(\Omega;\Rnn)}$.

The set $\Xpq(\Omega)$ is a reflexive Banach space, a sequence $(F_k,P_k)\in \Xpq(\Omega)$ converges weakly to $(F,P)\in \Xpq(\Omega)$, denoted $(F_k,P_k) \rightharpoonup (F,P)$, if and only if $F_k\weakto F$ in $L^p(\Omega;\Rnn)$ and $P_k\weakto P$ in $L^q(\Omega;\Rnn)$. The same holds for strong convergence. In addition, we shall require the following notions of convergence.

\begin{definition}[Div-curl convergence]\label{defdivcurlc}
We say that a sequence $(F_k, P_k)\in \Xpq(\Omega)$ is \divcurl-convergent
to $(F,P)$ if it converges weakly in $\Xpq(\Omega)$ and additionally
\begin{equation}\label{eqdivcurlconv}
 \begin{split}
  \curl F_k\to \curl F & \text{ strongly in } W^{-1,p}(\Omega;\Rnnn) ,\\
  \div P_k\to \div P & \text{ strongly in } W^{-1,q}(\Omega;\R^n) .
 \end{split}
\end{equation}
\end{definition}
As usual, the distributional differential operators $(\curl F)_{ijk}:=\partial_k F_{ij}-\partial_jF_{ik}$ and $(\div F)_i:=\sum_j \partial_j F_{ij}$ are both defined rowwise.

\begin{remark}[Div-curl convergence]
It follows from Rellich's theorem that, if $(F_k,P_k)\weakto (F,P)$ and
\begin{equation}
    \sup_k\|  \curl F_k\|_{L^p(\Omega)}+\|\div P_k\|_{L^q(\Omega)}<\infty ,
\end{equation} then $(F_k,P_k)$ \divcurl-converges to $(F,P)$.
\end{remark}

\begin{definition}[$\Delta$-convergence]\label{defdeltac}
A sequence $(y_k,z_k)\in \Xpq(\Omega)\times\Xpq(\Omega)$ is said to $\Delta$-converge to $(y,z)\in \Xpq(\Omega)\times\Xpq(\Omega)$ if $y_k\rightharpoonup y$ weakly, $z_k\rightharpoonup z$ weakly and $y_k-z_k\to y-z$ strongly in $\Xpq(\Omega)$.
\end{definition}

Evidently, one of the two weak convergences can be left out of the definition, since it follows from the other and the strong convergence of the difference. As in \cite{ContiMuellerOrtiz2018-Datadriven}, we choose the present variant of the definition for symmetry.

We proceed by identifying subsets of admissible states, i.~e., states that satisfy compatibility and the balance laws.

\begin{definition}[Constraint sets]\label{defineE}
We denote by $\Eg$ the set of states $(F,P)\in \Xpq(\Omega)$ such that
there is $u\in W^{1,p}(\Omega;\R^n)$ with
\begin{subequations}\label{sIes1A}
\begin{align}
    &
    F =  \nabla u  ,
    &  \text{in } \Omega ,
    \\ & \label{eqGammaD}
    u = g_D ,
    &   \text{on } \Gamma_D ,
\end{align}
\end{subequations}
and
\begin{subequations}\label{fRoa1l}
\begin{align}
    &
    {\rm div} P + f = 0 ,
    &
    \text{in } \Omega ,
    \\ &
    P(x) \nu = h_N ,
    &  \label{eqGammaN}
    \text{on } \Gamma_N.
\end{align}
\end{subequations}
We further define $\Emb:=\{(F,P)\in \Eg: FP^T\in\Rnsym\,\, \text{a.~e.~in } \Omega\}$.
\end{definition}

We note that $\Eg$ is an affine subspace of states that are compatible and in force equilibrium. The set $\Emb$ is the subset of states in $\Eg$ that are also in moment equilibrium.

As in \cite{ContiMuellerOrtiz2018-Datadriven}, the boundary conditions (\ref{eqGammaD}) and (\ref{eqGammaN}) are interpreted in the sense of traces. Details of the definitions of the spaces are given in Appendix \ref{appendixtraces}. Specifically, $u= g_D$ on $\Gamma_D$ means that the trace of $u$, as a function in $W^{1/q,p}(\partial\Omega;\R^n)\subseteq L^p(\partial\Omega;\R^n)$, equals $g_D$ for $\calH^{n-1}$-almost every point in $\Gamma_D$. The condition $P \nu = h_N$ on $\Gamma_N$ in turn means that
\begin{equation}
    \langle B_\nu P, \psi\rangle = \langle h_N,\psi\rangle ,
\end{equation}
for any $\psi\in W^{1/q,p}(\partial\Omega;\R^n)$ that obeys $\psi=0$ $\calH^{n-1}$-almost everywhere on $\partial\Omega\setminus\Gamma_N$.
Here $\langle \cdot,\cdot\rangle$ denotes the duality pairing between $ W^{1/q,p}(\partial\Omega;\R^n)$ and $ W^{-1/q,q}(\partial\Omega;\R^n)$, 
{and $B_\nu$ is the operator giving the normal component of the trace},
as in Lemma \ref{lemmatraceeq}.

The topologies of $\Delta-$ and \divcurl-convergence are related as follows.
We stress that here we assume that the limit has the form $(y,y)$ for some $y\in \Xpq(\Omega)$.
\begin{lemma}\label{lemmadeltatodivcurl}
Assume that $(y_k, z_k)\in \Xpq(\Omega)\times\Xpq(\Omega)$ is $\Delta$-convergent to $(y,y)$ and $(z_k) \subseteq \Eg$. Then, $y_k$ is \divcurl-convergent to $y$.
\end{lemma}
\begin{proof}
Let $y_k=(F_k',P_k')$, $z_k=(F_k,P_k)$. Since $z_k\in\Eg$, it follows that $\curl F_k=0$ and $\div P_k=f$ in the sense of distributions. By $\Delta$-convergence, we have $F_k'-F_k\to0$ strongly in $L^p(\Omega;\Rnn)$ and $P_k'-P_k\to0$ strongly in $L^q(\Omega;\Rnn)$. Therefore, $\curl F_k'=\curl(F_k'-F_k)\to0$ strongly in $W^{-1,p}(\Omega;\Rnnn)$ and $\div P_k'=f+ \div(P_k'-P_k)\to f=\div P$ strongly in $W^{-1,q}(\Omega;\R^n)$.
\end{proof}

We observe that, in this language, the classical div-curl Lemma \cite{Murat1978,  Tartar1978, Tartar1979, Murat1981, Tartar1982} takes the following form.

\begin{lemma}[\divcurl\ Lemma]\label{lemmadivcurl}
Let $(F_k,P_k)\in \Xpq(\Omega)$ be a sequence \divcurl-converging to $(F,P)$, with $1/p+1/q=1$. Then, 
    $F_k P_k^T \weakto FP^T$ in the sense of distributions
and, in particular, $F_k\cdot P_k\weakto F\cdot P$. Furthermore, if $p\ge n$, $\det F_k\weakto \det F$ in the sense of distributions and, if $p\ge n-1$, $\cof F_k\weakto \cof F$ in the sense of distributions.
\end{lemma}

Here and subsquently, we use the symbol $\cdot$ to denote the Euclidean scalar product in all $\R^N$ spaces. For example, if $a,b\in\R^n$, $a\cdot b:=\sum_i a_i b_i$ and, if $a,b\in\Rnn$, $a\cdot b:=\sum_{ij} a_{ij}b_{ij}=\Tr a^Tb$.

\begin{proof}
The assertions follow readily from \cite[Theorem 2]{Murat1978}.
\end{proof}

We shall assume that material behavior is characterized by a material data set $\calD \subseteq \Xpq(\Omega)$ collecting all the possible states attainable by the material. An important class of materials are those which can be characterized locally.

\begin{definition}[Local materials]\label{deflocalmat}
A material is said to be local if there is a local material data set $\calD_\loc \subseteq \Rnn\times\Rnn$ such that $\calD = \{(F,P) \in \Xpq(\Omega) \, : \, (F(x),P(x)) \in \calD_\loc \text{ {for} a.~e.~ {$x\in$} } \Omega\}$.
\end{definition}

The local material data set $\calD_\loc$ is subject to the physical requirement of material-frame indifference.

\begin{definition}[Material-frame indifference]
We say that $\calD_\loc\subseteq\Rnn\times\Rnn$ is material-frame indifferent if $(QF,QP)\in \calD_\loc$ for every $(F,P)\in\calD_\loc$ and $Q\in \SO(n)$.
\end{definition}

\begin{remark}[Orbit representation]\label{remarkorbit}
The set $\calD_\loc\subseteq\Rnn\times\Rnn$ is material-frame indifferent if and only if it consists of orbits under the left action of $\SO(n)$. Equivalently, there is $\mathcal{U}_\loc \subseteq\Rnsym\times \Rnn$ such that $\calD_\loc=\{(QF,QP): Q\in\SO(n), \,(F,P)\in \mathcal{U}_\loc\}$. In general, it is not possible to restrict $\mathcal{U}_\loc\subseteq\Rnsym\times \Rnsym$.
\end{remark}

We additionally recall from Definition~\ref{defineE} that the equation of moment equilibrium is local and algebraic and, therefore, it can be used to further constrain the local material data set.

\begin{definition}[Moment equilibrium]\label{defmomequil}
We say that $\calD_\loc\subseteq\Rnn\times\Rnn$ satisfies moment equilibrium if $PF^T = (PF^T)^T \in \Rnsym$ for every $(F,P)\in \calD_\loc$.
\end{definition}

Another important example of local material data set concerns sets that, as in classical elasticity, take the form of graphs.

\begin{definition}[Graph local material data sets]
Let $T:\Rnn\to\Rnn$ be a function mapping deformations to stresses. Then, the classical material data set defined by $T$ is
\begin{equation}\label{eqdefDlocT}
 \DlocT:=\{(F,T(F)): F\in\Rnn\}.
\end{equation}
\end{definition}

A further special example consists of maps of the form $T=DW$, i.~e., maps that derive from strain-energy density functions $W:\Rnn\to\R$.

For graph material data sets, the relation between material-frame indifference and moment equilibrium is as follows.

\begin{lemma}
Let $T:\Rnn\to\Rnn$ be a stress-deformation function.
\begin{enumerate}
\item Suppose that $T(Q\xi)=QT(\xi)$ for all $Q\in\SO(n)$, $\xi\in\Rnn$. Then, $\DlocT$ is material-frame indifferent.
\item If, in addition, $T(F)F^T\in\Rnsym$ for all $F\in\Rnn$, then $\DlocT$ satisfies moment equilibrium.
\item If $T=DW$, with $W:\Rnn\to\R$ such that $W(QF)=W(F)$ for all $F\in\Rnn$ and $Q\in\SO(n)$, then $T$ obeys both stated properties, $\DlocT$ is material-frame indifferent and satisfies moment equilibrium.
\end{enumerate}
 \end{lemma}
\begin{proof}
All assertions follow immediately from the definitions.
\end{proof}

We now turn to the formulation of the Data-Driven problem of finite elasticity. In order to measure deviations from the material data set in phase space, we introduce a function $V:\Rnn\to\R$ with the following properties.

\begin{assumption}[Deviation function]\label{Ze6How}
Let $V:\Rnn\to[0,\infty)$ be convex with $V(0)=0$ and let $V^*(\eta):=\sup_\xi \xi \cdot\eta-V(\xi)$ denote the convex conjugate. In addition,
\begin{equation}\label{eqassVVst}
    c_p |\xi|^p\le V(\xi)
    \text{ and }
    c_q |\eta|^q\le V^*(\eta)
    \text{ for all  } \xi,\eta\in\Rnn,
\end{equation}
with $1/p+1/q=1$ as in Assumption~\ref{assumptiongeneral} and {for} some constants $c_p > 0$, $c_q > 0$.
For a {non-empty} local data set $\calD_\loc\subseteq\Rnn\times \Rnn$, we further define
$\psi=\psi_{\calD_\loc}:\Rnn\times \Rnn\to[0,\infty)$ by
\begin{equation}\label{SP5voKb}
    \psi_{\calD_\loc}(F,P):=\min\{ V(F-F')+V^*(P-P') \, : \, (F',P') \in \calD_\loc \}.
\end{equation}
\end{assumption}
For example, $V(\xi)=\frac1p|\xi|^p$ and $V^*(\eta)=\frac1q |\eta|^q$.  From these properties, it follows that the function $\psi$ {is} non-negative, grows away from $\calD_\loc$ and $\psi = 0$ if and only if $(F,P) \in \calD_\loc$. Thus, the function $\psi$ provides a measure of the deviation of local states from $\calD_\loc$.

\section{Compactness}\label{deReW8}

The work function $F\cdot P$, which has already appeared in the div-curl Lemma, plays a crucial role in the compactness proof. The use of \divcurl-convergence permits to phrase the compactness properties exclusively in terms of the local data set $\calD_\loc$, without resorting to the constraint sets $\Emb$ or $\Eg$, which depend on the boundary data and the forcing. Indeed, one key insight is that the quantity $F\cdot P$ is a null Lagrangian with respect to \divcurl-convergence. In contrast, the approach in \cite{ContiMuellerOrtiz2018-Datadriven}, formulated in terms of transversality, involved both the data set and the constraint set and, therefore, indirectly the boundary data and external forces.

\begin{definition}[Coercivity of the material data set]\label{defcoercive}
We say that a material data set $\calD_\loc\subseteq\Rnn\times\Rnn$ is $(p,q)$-coercive if there are $c_F, c_P, c>0$ such that
\begin{equation}\label{eqdatacoercive}
    \frac1{c_F} |F|^p+ \frac1{c_P}|P|^q-c
    \le F\cdot P , \hskip1cm\text{ for all } (F,P)\in \calD_\loc.
\end{equation}
\end{definition}
\begin{remark}\label{remTcoercive}
If  $\calD_\loc$ is the graph of $T:\Rnn\to\Rnn$, coercivity means that
\begin{equation}\label{eqTcoercive}
    \frac1{c_F} |\xi|^p+ \frac1{c_P}|T(\xi)|^q-c
    \le \xi\cdot T(\xi) , \hskip1cm\text{ for all } \xi\in\Rnn.
\end{equation}
If $T$ is continuous, then it obviously suffices to show that there are $c_F,c_P,c>0$ and $R>0$ such that (\ref{eqTcoercive}) holds for any $\xi$ with $|\xi|>R$.
\end{remark}
Below we discuss (Lemma \ref{le:coercive} and Example \ref{example2dcoerciv}) examples of data sets in both two and three spatial dimensions that are $(p,q)$-coercive and material frame indifferent.

\begin{lemma}\label{lemmacoercFP}
Under Assumption~\ref{assumptiongeneral} and Definition~\ref{defineE}, the sets $\Emb$ and $\Eg$ are closed with respect to weak convergence in $\Xpq(\Omega)$. Furthermore,
\begin{equation}\label{lemmacoercFPeq}
    \int_\Omega F\cdot P \,dx \le c_\calE\|(F,P)\|_{\Xpq(\Omega)}+c_\calE ,
\end{equation}
for all $(F,P)\in \Eg$. The constant $c_\calE$ may depend on the functions $h_N$, $g_D$, and $f$ entering the definition of $\Eg$, see Definition.~\ref{defineE}, as well as on $n$, $p$, $q$, $\Omega$, $\Gamma_D$ and $\Gamma_N$.
\end{lemma}

\begin{proof}
We start by proving (\ref{lemmacoercFPeq}). Let $u\in W^{1,p}(\Omega;\R^n)$ be as in the definition of $\Eg$. By Poincar\'e's inequality,
{(\ref{eqGammaD}) and the fact that $\calH^{n-1}(\Gamma_D)>0$,}
\begin{equation}\label{eqpoincareuf}
 \|u\|_{W^{1,p}(\Omega)}\le c \|g_D\|_{W^{1/q,p}(\partial\Omega)} + c
\|Du\|_{L^p(\Omega)},
\end{equation}
with a constant that depends only on $\Omega$ {and $\Gamma_D$}. We compute, using 
{Definition \ref{defineE}} and
(\ref{eqdualitybnub}) {of} Lemma \ref{lemmatraceeq},
\begin{align}
\begin{split}
\int_\Omega F\cdot P\, dx &= \int_\Omega Du\cdot P \,dx \\
&= \langle B_\nu P, Bu\rangle -\int_\Omega u\cdot \div P \, dx\\
&= \langle B_\nu P, Bu-g_D\rangle+\langle B_\nu P, g_D\rangle +\int_\Omega
u\cdot f \, dx\\
&= \langle h_N, Bu-g_D\rangle+\langle B_\nu P, g_D\rangle +\int_\Omega
u\cdot f \, dx
\end{split}
\end{align}
where we additionally use that $Bu-g_D=0$ on $\Gamma_D$, which up to an $\calH^{n-1}$-null set is the same as $\partial\Omega\setminus\Gamma_N$, in order to insert the boundary condition $h_N$
(the meaning of the boundary conditions was explained after
Definition~\ref{defineE}). Estimating all terms using duality, Lemma \ref{lemmatracew1p} and Lemma \ref{lemmatraceeq} gives
\begin{align}
\begin{split}
\int_\Omega F\cdot P \,dx
\leq& \|h_N\|_{W^{-1/q,q}(\partial\Omega)}
\|Bu-g_D\|_{W^{1/q,p}(\partial\Omega)} \\
&+
\|B_\nu P\|_{W^{-1/q,q}(\partial\Omega)}
\|g_D\|_{W^{1/q,p}(\partial\Omega)} \\
&+
\|u\|_{L^p(\Omega)} \|f\|_{L^q(\Omega)}\\
\leq & \|h_N\|_{W^{-1/q,q}(\partial\Omega)}
(\|g_D\|_{W^{1/q,p}(\partial\Omega)} +
\|u\|_{W^{1,p}(\Omega)}) \\
&+
(\|P\|_{L^{q}(\Omega)}+\|f\|_{L^{q}(\Omega)})
\|g_D\|_{W^{1/q,p}(\partial\Omega)} \\
&+
\|u\|_{L^p(\Omega)} \|f\|_{L^q(\Omega)}\\
\leq & c_\calE + c_\calE \|(F,P)\|_{\Xpq(\Omega)},
\end{split}
\end{align}
where in the last step we have used (\ref{eqpoincareuf}).
This concludes the proof of (\ref{lemmacoercFPeq}).

It remains to prove that $\Eg$ and $\Emb$ are weakly closed. Consider a sequence $(F_k,P_k)\in \Eg$ converging weakly to $(F,P)$ in $\Xpq(\Omega)$. Let $(u_k)$ be the corresponding sequence of displacements. By (\ref{eqpoincareuf}), the sequence $u_k$ is bounded in $W^{1,p}(\Omega;\R^n)$. After extracting a subsequence, we can assume that $u_k\weakto u$ in $W^{1,p}(\Omega;\R^n)$, with $Du=F$. Continuity of the trace operator implies that the traces also converge weakly, $Bu_k\weakto Bu$, in $L^p(\partial\Omega;\R^n)$. Since $Bu_k=g_D$ as an $L^p$ function on the relatively open set $\Gamma_D$, the same holds for $u$.

From $\div P_k\weakto \div P$ in the sense of distributions and $\div P_k=f$ for all $k$, we deduce $\div P=f$ in the sense of distributions. It remains to show that $P$ obeys the required boundary condition on $\Gamma_N$. To this end, we fix $\psi\in W^{1/q,p}(\partial\Omega;\R^n)$ and observe that
\begin{equation}
    v\mapsto  \langle B_\nu v,\psi\rangle
\end{equation}
defines a linear continuous mapping on $E^q(\Omega)$
(see Definition~\ref{defEqomega} and Lemma~\ref{lemmatraceeq}).
In particular, weak convergence $P_k\weakto P$ implies
\begin{equation}
    \langle B_\nu P,\psi\rangle =
    \lim_{k\to\infty} \langle B_\nu P_k,\psi\rangle.
\end{equation}
Assume now that $\psi=0$ almost everywhere on $\Gamma_D$. Then,
\begin{equation}
    \langle B_\nu P,\psi\rangle =
    \lim_{k\to\infty} \langle B_\nu P_k,\psi\rangle=
    \lim_{k\to\infty} \langle h_N ,\psi\rangle.
\end{equation}
Therefore, $P$ satisfies the boundary condition and $(F,P)\in\Eg$.

By the div-curl Lemma, Lemma~\ref{lemmadivcurl}, we additionally obtain
\begin{equation}
    F_kP_k^T\weakto FP^T \text{ in the sense of distributions.}
\end{equation}
In particular, if $F_k P_k^T\in\Rnsym$ almost everywhere for all $k$, then the same holds for $FP^T$. This implies that $\Emb$ is also closed.
\end{proof}

\begin{lemma}\label{lemmacoerc}
Let Assumption~\ref{assumptiongeneral} and Assumption~\ref{Ze6How} be in force. Let $\calD_\loc$ be
$(p,q)$-coercive in the sense of (\ref{eqdatacoercive}). Then,
\begin{equation}\label{9AdRus}
\begin{split}
    &
    \|F\|_{L^p(\Omega)}^p+
    \|P\|_{L^q(\Omega)}^q+
    \|F'\|_{L^p(\Omega)}^p+
    \|P'\|_{L^q(\Omega)}^q
     \\ &
   \hskip8mm \le c_\calE+c_\calE
    \int_\Omega \left( V(F-F')+V^*(P-P')\right)dx ,
\end{split}
\end{equation}
for all $(F,P)\in\Eg$ and $(F',P')\in \calD$.
\end{lemma}

\begin{proof}
We write
\begin{equation}
     A:=\int_\Omega \left(V(F-F')+V^*(P-P')\right)dx
\end{equation}
for the {second term on the} right-hand side of (\ref{9AdRus}), and
\begin{equation}
     B:=\|F\|_{L^p(\Omega)}^p+\|P\|_{L^q(\Omega)}^q
\end{equation}
for the first two terms {on} the left-hand side.

From (\ref{eqassVVst}), we obtain
\begin{align}\label{eqf1p1fpsdf}
    \|F'-F\|_{L^p(\Omega)}^p+
    \|P'-P\|_{L^q(\Omega)}^q
    \le c A,
\end{align}
which immediately gives $ \|F'\|_{L^p(\Omega)}^p+ \|P'\|_{L^q(\Omega)}^q \le c (A+B)$.
At the same time, from (\ref{eqf1p1fpsdf}) and $|a+b|^p\le 2^{p-1} |a|^p+2^{p-1}|b|^p$ we have
\begin{equation}\label{eqbfpbppp}
\begin{split}
    B
    & =
    \|F\|_{L^p(\Omega)}^p+\|P\|_{L^q(\Omega)}^q
   \le cA+ 2^p\|F'\|_{L^p(\Omega)}^p+2^q\|P'\|_{L^q(\Omega)}^q.
\end{split}
\end{equation}

By Hölder's inequality,
\begin{align}\label{eqlemmacoerce3}
\begin{split}
    \int_\Omega (F'\cdot P'&-F\cdot P) dx
    \\ &
    =\int_\Omega (F'-F)\cdot P'+F\cdot (P'-P) dx
     \\ &
    \le\|F'-F\|_{L^p(\Omega)}
    \|P'\|_{L^q(\Omega)}+\|F\|_{L^p(\Omega)}\|P'-P\|_{L^q(\Omega)}
     \\ &
   \le cA+c A^{1/p} B^{1/q}+c A^{1/q} B^{1/p}.
\end{split}
\end{align}
By coercivity of $\calD_\loc$,
\begin{align}\label{eqlemmacoerce4}
    \|F'\|_{L^p(\Omega)}^p + \|P'\|_{L^q(\Omega)}^q
    \le
    c \int_\Omega F'\cdot P' \,dx +c.
\end{align}
By Lemma \ref{lemmacoercFP},
\begin{align}\label{eqlemmacoerce2}
    \int_\Omega F\cdot P \,dx
    \le
    c(\|F\|_{L^p(\Omega)}+\|P\|_{L^q(\Omega)}+1)\le c (B^{1/p}+B^{1/q}+1).
\end{align}
Collecting (\ref{eqlemmacoerce4}),  (\ref{eqlemmacoerce3}), (\ref{eqlemmacoerce2}), {in this order,} we obtain
\begin{equation}
\begin{split}
    \|F'\|_{L^p(\Omega)}^p&+\|P'\|_{L^q(\Omega)}^q
    \le 
    c \int_\Omega F'\cdot P' \,dx+c
     \\ &
   \le cA+c A^{1/p} B^{1/q}+c A^{1/q} B^{1/p}+c+
    c\int_\Omega F\cdot P \,dx
     \\ &
   \le cA+c A^{1/p} B^{1/q}+c A^{1/q} B^{1/p}+cB^{1/p}+cB^{1/q}+c.
\end{split}
\end{equation}
Recalling (\ref{eqbfpbppp}) we have, {by the inequality $ab\le \eps a^p+c_\eps b^q$,}
\begin{equation}
\begin{split}
    B
&   \le cA+ 2^p\|F'\|_{L^p(\Omega)}^p+2^q\|P'\|_{L^q(\Omega)}^q
     \\ &
   \le cA+c A^{1/p} B^{1/q}+c A^{1/q} B^{1/p}+cB^{1/q}+cB^{1/p}+c.
\end{split}
\end{equation}
Therefore, $B\le c +cA$, which concludes the proof.
\end{proof}

We are finally in a position to summarize our results and present a general compactness statement.
\begin{theorem}[Compactness]\label{theoremcompactness}
Let Assumption~\ref{assumptiongeneral} and Assumption~\ref{Ze6How} be in force. Let $\calD_\loc$ be
 $(p,q)$-coercive in the sense of (\ref{eqdatacoercive}). Let $(F_k,P_k)\in \Eg$ and $(F_k',P_k')\in \calD$ be such that
\begin{equation}
    \alpha:= \liminf_{k\to\infty} \int_\Omega \left(V(F_k-F_k')+V^*(P_k-P_k')\right) dx <\infty.
\end{equation}
Then:
\begin{itemize}
\item[i)] After extracting a subsequence, $(F_k,P_k)$ $\weakto$ $(F,P)$ and $(F_k',P_k')$ $\weakto$ $(F',P')$ weakly in $\Xpq(\Omega)$, with $(F,P)\in\Eg$.
\item[ii)] If, additionally, $\alpha=0$, then $F=F'$, $P=P'$, the sequences $\Delta$-converge to $((F,P),(F,P))$ and $(F_k',P_k')$ \divcurl-converges to $(F,P)$.
\item[iii)] If $(F_k,P_k)\in \Emb$ for all $k$, then $(F,P)\in\Emb$.
\end{itemize}
\end{theorem}

We remark that in this compactness statement moment equilibrium is encoded into $\Emb$, not $\calD$.
If $\calD_\loc$ is \divcurl-closed then also an encoding in $\calD$ is possible, see Section~\ref{V7bisp}. This permits to work with the affine space $\Eg$, which is simpler to approximate or discretize.

\begin{proof}
By Lemma \ref{lemmacoerc}, the sequences $(F_k, P_k)$ and $(F_k', P_k')$ are uniformly bounded in $\Xpq(\Omega)$. Therefore, there is a common subsequence that converges weakly to some limits $(F,P)$ and $(F', P')$. The {first assertion in} Lemma \ref{lemmacoercFP} shows that $(F,P)\in\Eg$.

If $\alpha=0$, then by (\ref{eqassVVst}) we have that $F_k-F_k'\to0$ strongly in $L^p$ and $P_k-P_k'\to0$ strongly in $L^q$, which implies $\Delta$-convergence.
\Divcurl-convergence follows then {from} Lemma \ref{lemmadeltatodivcurl}. Indeed, since $\curl F_k=\curl F=0$ and $F_k'-F_k\to0$ strongly in $L^p$, we have
\begin{equation}
    \curl (F_k'-F)=\curl (F_k'-F_k) \to0 \text{ strongly in }
    W^{-1,p}(\Omega;\Rnnn) .
\end{equation}
Analogously, from $\div P_k=\div P=f$ and $P_k'-P_k\to0$ we obtain
\begin{equation}
    \div (P_k'-P)=\div (P_k'-P_k) \to0 \text{ strongly in } W^{-1,q}(\Omega;\R^n),
\end{equation}
which proves the \divcurl-convergence of $(F_k', P_k')$.

Finally, if $(F_k, P_k)\in \Emb$ the second part of Lemma \ref{lemmacoercFP} shows that $(F,P)\in \Emb$.
\end{proof}

\def\ceins{b}
\def\czwei{d}
In the rest of this section we provide examples of material-frame-indifferent sets that obey the stated conditions.
\begin{lemma}  \label{le:coercive}
Let $g\in C^1(\R)$ be convex and such that  for some ${\ceins}, {\czwei}\ge0$,
\begin{equation}\label{eqgcoerc}
    |g'(t)|\le {\ceins}+{\czwei}|t|\hskip1cm\text{ for all $t\in\R$.}
\end{equation}
\begin{enumerate}
\item Assume that $n=2$, $a>0$, ${\ceins}>0$ and  $0\le {\czwei}< 2a$. Then, the  material data set generated by
\begin{equation}
    W_2(\xi):=\frac12|\xi|^2+\frac14 a|\xi|^4 + g(\det \xi)
\end{equation}
is $(4,4/3)$-coercive, in the sense that $T_2:=DW_2$ obeys (\ref{eqTcoercive}) with $p=4$, $q=4/3$.
\item
Assume that $n=3$, $a\ge 0$, $ e>0$, ${\ceins}>0$ and $0\le {\czwei}< 3e$. Then, the  material data set generated by
\begin{equation}
    W_3(\xi):=\frac12|\xi|^2+\frac14 a|\xi|^4 +\frac16 e|\xi|^6 + g(\det \xi)
\end{equation}
is $(6,6/5)$-coercive.
\end{enumerate}
\end{lemma}
Before proving the Lemma we show a concrete example.
\begin{example}\label{example2dcoerciv}
For all $a>0$ and $\beta\in(0,2a)$, the function $\hat W_2: \R^{2\times 2}\to\R$ defined by
\begin{equation}\label{eqdefhatw2}
    \hat W_2(\xi):=\frac12|\xi|^2+\frac14 a|\xi|^4 + \frac12 \beta (\det
    \xi-1-\frac{1+2a}{\beta})^2
\end{equation}
is invariant under rotations, $(4,4/3)$-coercive, and minimized by matrices in $\SO(2)$.
\end{example}

\begin{proof} The first two assertions are obvious from the definition and the previous lemma
using $g(t)=\frac12\beta(t-1-\frac{1+2a}{\beta})^2$, ${\ceins}=1+2a+\beta$,
${\czwei}=\beta$. To check the third, let $\xi\in\R^{2\times 2}$, and choose rotations $Q,R\in\SO(2)$ and $x,y\in\R$ such that $\xi=Q\diag(x,y)R$. Then,
\begin{equation}
\begin{split}
    \hat W_2(\xi)
    & =
    \hat W_2(\diag(x,y))
    \\ & =
    \frac12 (x^2+y^2)+\frac14 a (x^2+y^2)^2 +
    \frac12\beta(xy-1-\frac{1+2a}{\beta})^2.
\end{split}
\end{equation}
If $xy<0$, then $\hat W_2(\diag(x,y))>\hat W_2(\diag(|x|,|y|))$. Therefore, minimizers have $x,y\ge 0$. If $x\ne y$, the arithmetic-geometric mean inequality gives $\hat W_2(\diag(x,y))> \hat W_2(\diag((xy)^{1/2},(xy)^{1/2}))$. We are left with the case $x=y\ge0$, and
\begin{equation}
\begin{split}
    \hat W_2(\xi)&=\hat W_2(\diag(x,x))
    =x^2+a x^4 + \frac12\beta(x^2-1-\frac{1+2a}{\beta})^2\\
    &= (a+\frac\beta2) (x^2-1)^2+\frac2\beta (a+\frac12)^2+a+1.
\end{split}
\end{equation}
Hence, the minimum is attained if and only if $x=y=1$ and $\xi=QR\in\SO(2)$.
\end{proof}

\begin{proof}[Proof of Lemma~\ref{le:coercive}]
For $n=2$, we compute
\begin{equation}
    DW_2(\xi)=\xi+a|\xi|^2\xi+g'(\det\xi)\cof\xi
\end{equation}
and, recalling that $F\cdot\cof F=2\det F$,
\begin{equation}
\begin{split}
    \xi \cdot DW_2(\xi)&=|\xi|^2+a|\xi|^4+2 g'(\det\xi)\det \xi\\
    &\ge (a-\frac12 {\czwei}) |\xi|^4 + \frac12{\czwei}|\xi|^4 -2{\ceins}|\det\xi| -2{\czwei}
    |\det\xi|^2\\
    &\ge (a-\frac12 {\czwei}) |\xi|^4  -{\ceins}|\xi|^2 ,
\end{split}
\end{equation}
where we have used the growth of $g'$ and $|\xi|^2\ge 2\det \xi$. The last term can be absorbed in the fourth-order one if $\xi$ is sufficiently large. This proves that $|\xi|^4\le c \xi\cdot DW_2(\xi)+c$. With $|DW_2(\xi)|\le c|\xi|^{3}+c$, the proof is concluded.

The case $n=3$ is similar. Indeed, here $|\xi|^3\ge 3\det\xi$ and, therefore,
\begin{equation}
\begin{split}
    \xi \cdot DW_3(\xi)&=|\xi|^2+a|\xi|^4+e|\xi|^6+3 g'(\det\xi)\det \xi\\
    &\ge (e-\frac13 {\czwei}) |\xi|^6 + \frac13{\czwei}|\xi|^6 -3{\ceins}|\det\xi| -3{\czwei}
    |\det\xi|^2\\
    &\ge (e-\frac13 {\czwei}) |\xi|^6  -3{\ceins}|\det\xi| ,
\end{split}
\end{equation}
leading to $|\xi|^6\le c \xi\cdot DW_3(\xi)+c$. With $|DW_3(\xi)|\le c|\xi|^{5}+c$, the proof is concluded.
\end{proof}

\section{\divcurl-closed material data sets}\label{V7bisp}

\newcommand\nhg{g^*}
\newcommand\hatg{g}
\newcommand\nhf{f^*}
\newcommand\hatf{f}
\newcommand\hatnhg{\hat g^*}

In this section, we develop notions of closedness of the material data set conferring the Data-Driven problem (\ref{9lFapr}) sufficient lower-semicontinuity.
As in the case of coercivity, working with \divcurl-convergence permits us to formulate closedness in terms of the data set $\calD$ alone, without involving the boundary data and forcing terms.

\subsection{Locally and globally \divcurl-closed material data sets}

\begin{definition}[\divcurl-closed material data sets]\label{definitiondeltaclosed}
We say that a material data set $\calD\subseteq \Xpq(\Omega)$ is $(p,q)$-\divcurl-closed if the limits of all \divcurl-convergent sequences $(F_k,P_k)\in\calD\subseteq \Xpq(\Omega)$ are in $\calD$.
\end{definition}

We first show that this property can be localized, in the sense that the limits can be assumed to be constant.
{A related question on localization has been raised in \cite{RoegerSchweizerDD}.}

\begin{definition}[\divcurl-closed local material data sets]\label{definitionlocdeltaclosed}
We say that a local material data set $\calD_\loc\subseteq\Rnn\times\Rnn$ is locally $(p,q)$-\divcurl-closed if, for all bounded nonempty open sets $\omega\subseteq \R^n$, and
all $(F_*, P_*)\in\Rnn\times \Rnn$, the following holds:
if there is a sequence of pairs $(F_k,P_k)\in \Xpq(\omega)$ \divcurl-converging to the constant function $(F_*,P_*)$ and such that $(F_k, P_k)\in\calD_\loc$ almost everywhere, then $(F_*,P_*)\in \calD_\loc$.
\end{definition}

The set $\Xpq(\omega)$ is defined as in Definition \ref{defxpq}. If $\calD_\loc$ is locally \divcurl\ closed, then it is also closed as a subset of $\Rnn\times \Rnn$. However, the converse is not true, as the example $\{\pm e_1\otimes e_1\}\times \Rnn$ shows.

We first show that it suffices to consider a single instance for the domain $\omega$.

\begin{lemma}\label{lemmalocdivcclsetun}
 Let  $\calD_\loc\subseteq\Rnn\times\Rnn$, and fix an open bounded nonempty set $\hat\omega\subseteq\R^n$. Assume that
for any sequence $(F_k,P_k)\in\Xpq(\hat\omega)$ which \divcurl-converges to a constant function $(F_*,P_*)$  and obeys
$(F_k,P_k)\in\calD_\loc$ almost everywhere
one has $(F_*,P_*)\in \calD_\loc$. Then, $\calD_\loc$ is locally $(p,q)$-\divcurl-closed.
\end{lemma}
\begin{proof}
Fix an open bounded set $\omega$ and a sequence $(F_k,P_k)\in \Xpq(\omega)$ \divcurl-converging to a constant function $(F_*,P_*)$ and such that $(F_k, P_k)(x)$ $\in$ $\calD_\loc$ for almost every $x\in\omega$. We need to show that $(F_*, P_*)\in\calD_\loc$.

Since both $\omega$ and $\hat\omega$ are open and bounded, there are $\alpha>0$ and $\beta\in\R^n$ such that $\alpha\hat\omega+\beta\subseteq\omega$. We define, for $x\in\hat\omega$,
\begin{equation}
 \hat F_k(x):=F_k(\alpha x + \beta) \text{ and }
 \hat P_k(x):=P_k(\alpha x + \beta).
\end{equation}
Then, $(\hat F_k,\hat P_k)\in \Xpq(\hat\omega)$, and they converge weakly to the constant function $(F_*, P_*)$.

We next show that
\begin{equation}\label{eqcurlhatfkohst}
    \curl \hat F_k\to0, \quad \text{in } W^{-1,p}(\hat\omega;\Rnnn),
\end{equation} 
and
\begin{equation}\label{eqdivhatfkohst}
    \div \hat P_k\to0, \quad \text{in } W^{-1,q}(\hat\omega;\R^{n}) .
\end{equation}
To that end, we choose $\hatg_k\in L^p(\omega;\Rnnn)$ and $\nhg_k\in L^p(\omega;\Rnnnn)$ such that
\begin{equation}
    \curl F_k=\hatg_k + \div \nhg_k
\end{equation}
and
\begin{equation}
  \|\hatg_k\|_{L^p(\omega)} + \|\nhg_k\|_{L^p(\omega)}
  \le c \|\curl F_k\|_{W^{-1,p}(\omega)}\to0
\end{equation}
(see for example \cite[Th.~4.3.3]{Ziemer1989-weaklydiff-book}).
We then define, for $x\in\hat\omega$,
\begin{equation}
\text{$\hat \hatg_k(x):=\alpha\hatg_k(\alpha x+ \beta)$ and
$\hatnhg_k(x):=\nhg_k(\alpha x+ \beta)$,  }
\end{equation}
so that
$\curl\hat F_k=\hat\hatg_k+\div \hatnhg_k$. From
\begin{equation}
 \|\hat\hatg_k\|_{L^p(\hat\omega)} + \|\hatnhg_k\|_{L^p(\hat\omega)}
\le \alpha^{1-n/p}
  \|\hatg_k\|_{L^p(\omega)} + \alpha^{-n/p}\|\nhg_k\|_{L^p(\omega)}
 \to0
\end{equation}
we obtain (\ref{eqcurlhatfkohst}). The proof of (\ref{eqdivhatfkohst}) is analogous.

Therefore, $(\hat F_k,\hat P_k)$ is \divcurl-convergent to $(F_*, P_*)$, and $(\hat F_k, \hat P_k)\in \calD_\loc$ pointwise almost everywhere. The assumption then yields $(F_*, P_*)\in\calD_\loc$, which concludes the proof.
\end{proof}

\begin{proposition}[Equivalence of local and global closedness]\label{propdeltaclosed}
Assume that $\calD:=\{(F,P) \in \Xpq(\Omega): (F,P)\in\calD_\loc \text{ a.~e.}\}$ for some $\calD_\loc\subseteq\Rnn\times \Rnn$. Then, the set $\calD$ is \divcurl-closed  if and only if $\calD_\loc$ is locally \divcurl-closed.
\end{proposition}

\begin{proof}
Assume that $\calD$ is \divcurl-closed. Choose $(F_*,P_*)\in \Rnn\times\Rnn$. We can identify this pair with constant functions $(\hat F_*,\hat P_*) \in\Xpq(\Omega)$. Assume that there is a sequence $(F_k, P_k){\in \Xpq(\Omega)}$ that \divcurl-converges to $(\hat F_*,\hat P_*)$ as in Definition \ref{definitionlocdeltaclosed}. We need to show that $(F_*, P_*)\in \calD_\loc$. Since these sequences fulfill the properties stated in Definition \ref{definitiondeltaclosed}, the fact that $\calD$ is \divcurl-closed gives $(\hat F_*,\hat P_*)\in\calD$, which is equivalent to $(F_*, P_*)\in \calD_\loc$.

We now come to the difficult direction, in which we need to localize. Fix a sequence $(F_k,P_k)\in \calD\subseteq\Xpq(\Omega)$ \divcurl-converging to some $(F,P)\in \Xpq(\Omega)$. Assume that $\calD_\loc$ is locally \divcurl-closed. We need to show that $(F,P)\in \calD_\loc$ pointwise almost everywhere.

Since $\curl (F_k-F)\in W^{-1,p}(\Omega;\Rnnn)$, there are
$\hatg_k\in L^p(\Omega;\Rnnn)$ and
$\nhg_k\in L^p(\Omega;\Rnnnn)$ such that
\begin{equation}\label{eqcurlfkfggk}
    \curl (F_k-F)= \|\curl (F_k-F)\|_{W^{-1,p}(\Omega)} ( \hatg_k + \div \nhg_k) ,
\end{equation}
with
\begin{equation}
 \sup_k \|\hatg_k\|_{L^p(\Omega)} + \|\nhg_k\|_{L^p(\Omega)} < \infty
\end{equation}
(see for example \cite[Th.~4.3.3]{Ziemer1989-weaklydiff-book}). Correspondingly, there are
$\hatf_k\in L^q(\Omega;\R^n)$ and $\nhf_k\in L^q(\Omega;\Rnn)$ such that
\begin{equation}
    \div (P_k - P)= \|\div (P_k-P)\|_{W^{-1,q}(\Omega)} ( \hatf_k + \div \nhf_k) ,
\end{equation}
with
\begin{equation}
 \sup_k \|\hatf_k\|_{L^q(\Omega)} + \|\nhf_k\|_{L^q(\Omega)} < \infty.
\end{equation}
We define the measures $\mu_k$ on $\Omega$ as
\begin{equation}
 \mu_k:= (|F_k|^p+|P_k|^q+|\hatg_k|^p+|\nhg_k|^p+|\hatf_k|^q+|\nhf_k|^q)\calL^n.
\end{equation}
These measures are uniformly bounded, and hence, after taking a subsequence, have a weak limit $\mu$.

We fix $x_*\in\Omega$ such that $x_*$ is a Lebesgue point of $(F,P)$ and
\begin{equation}\label{eqdmudcall}
\frac{d\mu}{d\calL^n}(x_*)<\infty.
\end{equation}
In addition, we let $F_*:=F(x_*)$, $P_*:=P(x_*)$. Since these properties hold for $\calL^n$-almost any $x_*$, it suffices to show that $(F_*,P_*)\in\calD_\loc$.

By (\ref{eqdmudcall}) there are $r_0>0$, $C_0>0$ such that $\overline B_{r_0}(x_*)\subseteq\Omega$ and
\begin{equation}
 \mu(\overline B_r(x_*))< C_0r^n\hskip5mm\text{ for all $r\le r_0$. }
\end{equation}
For any function $\varphi\in L^p(\Omega;\Rnn)$ and $r<\dist(x_*, \partial\Omega)$, we define the blow-up $R_r\varphi\in L^p(B_1;\Rnn)$ by $(R_r\varphi)(y):=\varphi(x_*+ry)$. Then,
\begin{equation}\label{eqblowup}
    \int_{B_1}|R_r\varphi|^p dy
    =
    \frac{1}{r^n} \int_{B_r(x_*)} |\varphi|^p dz ,
\end{equation}
and the analogous identity holds for the $L^q$ norm.

We fix a norm $\|\cdot \|_{w,p}$ that induces the weak $L^p(B_1;\Rnn)$ topology on bounded {subsets of} $L^p(B_1;\Rnn)$ and analogously $\|\cdot\|_{w,q}$.

Fix any sequence $r_j\to0$, with $r_j\in (0,r_0)$ for all $j$. From $\mu_k \rightharpoonup \mu$, we have
\begin{equation}
    \limsup_{k\to\infty}
    \mu_k(B_{r_j}(x_*))\le \mu(\overline B_{r_j}(x_*)) < C_0 r_j^n
\end{equation}
and deduce that there is $k_j\ge j$ such that $\mu_k(B_{r_j}(x_*))\le C_0 r_j^n$ for all $k\ge k_j$. This implies, using (\ref{eqblowup}) and the definition of $\mu_k$,
\begin{equation}\label{eqboundedc0r}
\begin{split}
    & \|R_{r_j}F_{k}\|_{L^p(B_1)}^p+
    \|R_{r_j}P_{k}\|_{L^q(B_1)}^q+\|R_{r_j}\hatg_k\|_{L^p(B_1)}^p+\|R_{r_j} \hatf_k\|_{
    L^q(B_1)}^q
     \\ & \hskip1cm +
    \|R_{r_j}\nhg_k\|_{L^p(B_1)}^p+\|R_{r_j}\nhf_k\|_{
    L^q(B_1)}^q
    \le C_0 \text{ for all } k\ge
    k_j.
 \end{split}
\end{equation}
As $k\to\infty$, we have $(R_{r_j}F_{k}, R_{r_j}P_{k})-(R_{r_j}F, R_{r_j}P)\weakto 0$ in $\Xpq(B_1)$. Therefore, there is $\hat k_j\ge k_j$ such that
\begin{equation}\label{eqweakconvmed}
\begin{split}
    \|R_{r_j}F_{k}-R_{r_j}F\|_{w,p}+\|R_{r_j}P_{k}-R_{r_j}P\|_{w,q}&\le
    \frac1j \text{ for all } k\ge \hat k_j.
\end{split}
\end{equation}
We define $\hat F_j:=R_{r_j}F_{\hat k_j}\in L^p(B_1;\Rnn)$,  $\hat P_j:=R_{r_j}P_{\hat k_j} \in L^q(B_1;\Rnn)$. 
{Since} $x_*$ is a Lebesgue point of $(F,P)$ and $r_j\to0$,
\begin{equation}\label{eqrjfp}
    (R_{r_j}F,R_{r_j}P) \to (F_*, P_*) \text { strongly in } \Xpq(B_1).
\end{equation}
{Further,}
 $(\hat F_j,\hat P_j)\in\calD_\loc$ a.~e.~and, by (\ref{eqweakconvmed}) { and (\ref{eqboundedc0r})}, 
\begin{equation}\label{eq37weak}
    (\hat F_j, \hat P_j)-(R_{r_j}F,R_{r_j}P) \weakto 0
    \text { weakly in } \Xpq(B_1).
\end{equation}
which implies, with (\ref{eq37weak}),
\begin{equation}\label{eqhatfjpj}
    (\hat F_j, \hat P_j)\weakto (F_*,P_*) \text { weakly in } \Xpq(B_1).
\end{equation}
We recall that we identify $(F_*,P_*)=(F(x_*),P(x_*))\in\Rnn\times\Rnn$ with a constant function on $B_1$.

Furthermore,
\begin{equation}
\begin{split}
    \curl (\hat F_j-F_*) =& \curl (R_{r_j}(F_{\hat k_j}-F)) + \curl(R_{r_j} F-F_*)
    \\ = &
    r_j R_{r_j} \curl (F_{\hat k_j}-F) +  \curl(R_{r_j} F-F_*) .
\end{split}
\end{equation}
Equation (\ref{eqrjfp}) implies $R_{r_j} F-F_*\to0$ in $L^p(B_1;\Rnn)$ and, therefore, $\curl(R_{r_j} F-F_*)\to0$ in $W^{-1,p}(B_1;\Rnnn)$. At the same time, (\ref{eqcurlfkfggk}) gives
\begin{equation}
\begin{split}
    r_j R_{r_j} &\curl (F_{\hat k_j}-F)
     \\ &
    =\|  \curl (F_{\hat k_j}-F) \|_{W^{-1,p}(\Omega)} (r_j R_{r_j} \hatg_{\hat k_j} +
     \div R_{r_j}\nhg_{\hat k_j}).
\end{split}
\end{equation}
By (\ref{eqboundedc0r}), we have $\| R_{r_j} \hatg_{\hat k_j}\|_{L^p(B_1)}^p +  \| R_{r_j} \nhg_{\hat k_j}\|_{L^p(B_1)}^p\le C_0 $, hence $ (r_j R_{r_j} \hatg_{\hat k_j} +  \div R_{r_j}\nhg_{\hat k_j})$ is bounded in $W^{-1,p}(B_1;\R^{n\times n\times n})$. Recalling the assumption $\|  \curl (F_{\hat k_j}-F) \|_{W^{-1,p}(\Omega)}\to0$, we conclude that
\begin{equation}
    \curl (\hat F_j-F_*) \to0 \text{ strongly in } W^{-1,p}(B_1;\Rnnn).
\end{equation}
An analogous computation leads to
\begin{equation}\label{propdeltaclosedfin}
    \div(\hat P_j-P_*) \to0 \text{ strongly in } W^{-1,q}(B_1;\R^n).
\end{equation}
Since $\calD_\loc$ is locally \divcurl-closed, this implies $(F_*,P_*)\in\calD_\loc$, which concludes the proof.
\end{proof}

\begin{example}
The set $\calD_\loc^\ge:=\{(F,P)\in\Rnn\times \Rnn: \det F\ge 0\}$ is $(p,q)$-\divcurl-closed for $p\ge n$. The set $\calD_\loc^>:=\{(F,P)\in\Rnn\times \Rnn: \det F> 0\}$ is not. If $\calD'_\loc$ and $\calD''_\loc$ are \divcurl-closed, then so is the intersection.
\end{example}
\begin{proof}
By the \divcurl~Lemma, Lemma \ref{lemmadivcurl},  if $(F_k, P_k)$ is \divcurl-convergent to $(F,P)$ then $\det F_k$ is weakly convergent to $\det F$. Therefore, $\det F_k\ge0$ implies $\det F\ge0$. For the second assertion, it suffices to consider $(F_k,P_k)(x):=(k^{-1}\Id,0)$. The third one is immediate from the definition.
\end{proof}

\subsection{Strict polymonotonicity and quasimonotonicity}

{In this Section we introduce suitable concepts of strict polymonotonicity and strict quasimonotonicity that can be used to establish that} specific material data sets generated by stress-strain functions are \divcurl-closed. We additionally discuss {explicit} examples in two and three dimensions.

{We denote by $M(G)\in \R^{\tau(n)}$ the vector of all minors of $G$, $M(G)=(G,\det G)\in \R^5$ if $n=2$ and $M(G)=(G,\cof G, \det G)\in \R^{19}$ if $n=3$.}
\begin{definition}[Strict polymonotonicity]\label{defpolymonotone}
We say that $T:\Rnn\to\Rnn$ is strictly polymonotone if there is a function $B\in C^0(\Rnn\times \Rnn;[0,\infty))$ such that $B(F,G)=0$ if and only if $G=0$ and
\begin{equation}\label{eqpolymonot}
    (T(F+G)-T(F))\cdot G \ge A(F)\cdot M(G)+B(F,G) ,
\end{equation}
for all $F,G\in\Rnn$ and some $A:\Rnn\to \R^{\tau(n)}$. 
\end{definition}

\begin{remark}\label{remarkdivcurlminors}
  By the \divcurl~Lemma, Lemma \ref{lemmadivcurl}, if $(F_k, P_k)$ is \divcurl-convergent to $(F,P)$ and $p\ge n$
  then $M(F_k)\weakto M(F)$ and $M(F_k-F)\weakto0$ in the sense of distributions.
\end{remark}

A more general condition is furnished by quasimonotonicity. However, this condition {is} difficult to verify in practice, unless some variant of polymonotonicity holds. In the present setting, quasimonotonicity can be defined as follows. 

\begin{definition}[Strict quasimonotonicity]\label{defquasimonotone} We say that a Borel measurable, locally bounded function $T:\Rnn\to\Rnn$ is strictly quasimonotone if there is a function $B\in C^0(\Rnn\times \Rnn;[0,\infty))$ such that $B(F,G)=0$ if and only if $G=0$ and
\begin{equation}\label{eqdefquasimon}
\int_\omega (T(F+D\varphi)-T(F))\cdot D\varphi \, dx \ge \int_\omega B(F,D\varphi) dx
\end{equation}
for all bounded Lipschitz sets $\omega\subseteq\R^n$,  all $F\in\Rnn$, and all $\varphi\in C^\infty_c(\omega;\R^n)$.
\end{definition}
We refer to \cite{Zhang1988-quasimonotone} for further discussion of quasimonotonicity and its relevance to the existence of solutions of the equilibrium equations $\div  T(\nabla u) = 0$.

\begin{lemma}\label{lemmapolymonquasimon}
Let  $T:\Rnn\to\Rnn$ be Borel measurable, locally bounded and strictly polymonotone. Then, {$T$} is strictly quasimonotone.
\end{lemma}
\begin{proof}
Let $\omega$, $\varphi$, $F$ be as in Definition \ref{defquasimonotone}. We compute, using (\ref{eqpolymonot}),
\begin{equation}
\begin{split}
    \int_\omega (T(F+D\varphi)-T(F))\cdot D\varphi \, dx \ge &
    A(F)\cdot \int_\omega M(D\varphi) dx + \int_\omega B(F,D\varphi) dx.
\end{split}
\end{equation}
The assertion then follows from $\int_\omega M(D\varphi)dx=0$.
\end{proof}

In order to pass from quasimonotonicity to closedness, we need, much as in the case of quasiconvexity, appropriate growth and continuity assumptions (see for example \cite[Th.~4.4 and 4.5]{MuellerLectureNotes} or \cite[Sect.~8.2.2]{DacorognaBuch2008}). {In} the remainder of the paper, we restrict attention to the case of $F\in L^p(\Omega;\Rnn)$ and $P\in L^q(\Omega;\Rnn)$, with $P=T(F)$ pointwise. We observe that $|T(F)|^q\sim |F|^p$ implies $|T(F)|\sim |F|^{p/q}=|F|^{p-1}$ (since $q=p/(p-1)$) and, therefore, $|DT(F)|\sim |F|^{p-2}$. These scalings arise also naturally from $W(F)\sim |F|^p$ and $T=DW$, $DT=D^2W$.

\begin{theorem}\label{theostrictlyqmclosed2}
Assume that
$p\ge n\ge 2$, $T:\Rnn\to\Rnn$. Assume one of the following:
\begin{enumerate}
 \item[(a)] $T$ is continuous and strictly polymonotone;
 \item[(b)]
$T$ is strictly quasimonotone and
\begin{equation}\label{eqtfgtf}
 |T(F+G)-T(F)|\le c (|F|^{p-2}+|G|^{p-2}+1) |G|
\end{equation}
for some $c>0$ and all $F,G\in\Rnn$.
\end{enumerate}
Then,
\begin{enumerate}
 \item $\DlocT$, {as defined in (\ref{eqdefDlocT}),} is locally $(p,q)$-\divcurl-closed;
  \item Assume $(F_k, P_k)\in \Xpq(\omega)$ is \divcurl-convergent to $(F,P)\in\Xpq(\omega)$ for some bounded Lipschitz nonempty $\omega\subseteq\R^n$, with $P_k=T(F_k)$ pointwise almost everywhere. Then, $(F_k,P_k)$ converges strongly in $L^1(\omega;\Rnn\times \Rnn)$ to $(F,P)$ and $P=T(F)$ pointwise almost everywhere.
\end{enumerate}
\end{theorem}


We remark that  (\ref{eqtfgtf}) implies that $T$ is continuous, and, therefore, Borel measurable and locally bounded.
\begin{proof}
\newcommand\omegaba{{B_1}}
\newcommand\omegar{{B_r}}
{
We first prove that in both cases the following local version of (ii):
\begin{enumerate}\addtocounter{enumi}{2}
 \item Assume $(F_k, P_k)\in \Xpq(\omegaba)$ is \divcurl-convergent to the constant function $(F_*,P_*)\in\Rnnnn$, with $P_k=T(F_k)$ pointwise almost everywhere. Then, $(F_k,P_k)$ converges strongly in $L^1(\omegaba;\Rnn\times \Rnn)$ to $(F_*,P_*)$ and $P_*=T(F_*)$.
 \end{enumerate}
}
Let $(F_k,P_k)\in X_{p,q}(\omegaba)$ be a sequence which \divcurl-converges to  {$(F_*,P_*)$} in  $\Xpq(\omegaba)$ and such that $P_k=T(F_k)$ almost everywhere, for any $k$.
By the div-curl Lemma \ref{lemmadivcurl}, we obtain $P_k\cdot F_k\weakto P\cdot F_*$ in the sense of distributions, and, therefore, $P_k\cdot(F_k-F_*)\weakto0$. Weak convergence additionally gives $T(F_*)\cdot (F_k-F_*)\weakto0$ and, therefore,
\begin{equation}\label{eqpkdw2fkfa}
    (P_k - T(F_*))  \cdot (F_k - F_*) \weakto 0 ,
\end{equation}
in the sense of distributions. 

In case (a), $T$ is strictly polymonotone. Let $A$ and $B$ be as in (\ref{eqpolymonot})  in the definition of strict polymonotonicity. From (\ref{eqpolymonot}) and Remark~\ref{remarkdivcurlminors} we obtain, for any test function $\theta\in C^\infty_c(\omegaba;[0,\infty))$,
\begin{equation}\label{eqbfkfsd}
\begin{split}
 0=&\limsup_{k\to\infty} \int_\omegaba   (T(F_k)- T({F_*}))  \cdot (F_k-{F_*}) \, \theta\, dx\\
 \ge& \limsup_{k\to\infty} \left[A({F_*})\cdot \int_\omegaba M(F_k-{F_*}) \, \theta\,dx + \int_\omegaba B({F_*},F_k-{F_*})\,\theta\, dx\right] \\
 =& \limsup_{k\to\infty} \int_\omegaba B({F_*},F_k-{F_*})\,\theta\, dx .
\end{split}
\end{equation}
{Together with} $B\ge0$, {this implies that} the sequence of functions $x\mapsto B({F_*},F_k(x)-{F_*})$ converges to zero in $L_\loc^1(\omegaba)$ and, therefore, (up to a subsequence) pointwise almost everywhere.  Since $B(F_*,\cdot)$ is continuous and $B(F_*,G)>0$ for $G\ne0$, we obtain that (up to a further subsequence), for almost every $x\in\omegaba$, either $F_k(x)\to {F_*}$ or $|F_k|(x)\to\infty$ holds. Since $F_k$ is bounded in $L^p(\omegaba;\Rnn)$, we obtain
\begin{equation}\label{eqfkfpwae}
    F_k\to {F_*}, \text{ pointwise a.~e.~along a subsequence. }
\end{equation}
{
By continuity of $T$, $P_k=T(F_k)\to T(F_*)$ along the same subsequence.
Since $(F_k,P_k)$ is bounded in $\Xpq(\omegaba)$, it is equiintegrable and $(F_k,P_k)\to (F_*,T(F_*))$ in $L^1(\omegaba;\Rnn)$ along the same subsequence.
By uniqueness of the limit the same holds without extracting a subsequence and $P_*=T(F_*)$. This concludes the proof of (iii) in case (a).}


In case (b), the proof is more complex. We define $G_k:=F_k-{F_*}$. Then, $G_k\weakto 0$ in $L^p(\omegaba;\Rnn)$ and $\curl G_k\to0$ in $W^{-1,p}(\omegaba;\Rnnn)$, therefore (by the Hodge decomposition) there is $\psi_k\in W^{1,p}(\omegaba;\R^n)$ such that $\psi_k\weakto0$ in $W^{1,p}(\omegaba;\R^n)$ and $G_k-D\psi_k\to0$ in $L^p(\omegaba;\Rnn)$. 

{
The sequence 
 $(P_k - T({F_*}))  \cdot (F_k - {F_*})$ is bounded in $L^1(\omegaba)$, and by (\ref{eqpkdw2fkfa}) it converges to zero weakly-$*$ in measures. For the same reason,
 passing to a subsequence $|(P_k - T({F_*}))  \cdot (F_k - {F_*})|$ 
 has a weak-$*$ limit in measures. As usual, this implies that along this subsequence
\begin{equation}\label{eqpkdw2fkfab}
\int_{\omegar}    (P_k - T({F_*}))  \cdot (F_k - {F_*})\, dx\to 0
\quad
\text{ for a.e. $r\in(0,1)$}.
\end{equation}}
{Similarly, since also $|D\psi_k|^p$ is bounded in 
 $L^1(\omegaba)$, it  has (up to a further subsequence) a weak-$*$ limit in measures, and 
\begin{equation}\label{eqpkdw2fkfabpsi}
\limsup_{\delta\to0} \limsup_{k\to\infty}\int_{\omegar \setminus B_{r-\delta}}|D\psi_k|^p    dx=0
\quad
\text{ for a.e. $r\in(0,1)$}.
\end{equation}
For the next part of the proof we fix $r\in(0,1)$ such that 
(\ref{eqpkdw2fkfab}) and (\ref{eqpkdw2fkfabpsi}) hold.}
{
We first show that 
we can choose $\varphi_k$ such that 
\begin{equation}\label{eqvarphik}
\varphi_k\in C^\infty_c({\omegar};\R^n) \text{ and  } \varphi_k-\psi_k\to0 \text{ in } W^{1,p}({\omegar};\R^n).
\end{equation}
To do this we define $\ell_k:=\|\psi_k\|_{L^p(\omegar)}^{1/2}\to0$
and $\hat\psi_k(x):=\psi_k(x) \min\{ \ell_k^{-1} (r-|x|), 1\}$. Then $\hat\psi_k\in W^{1,p}_0(\omegar;\R^n)$ 
and from (\ref{eqpkdw2fkfabpsi}) we obtain
$\hat\psi_k-\psi_k\to0$ strongly in $W^{1,p}(\omegar;\R^n)$.
By density, there is $\varphi_k$ as stated in (\ref{eqvarphik}).} 

We define $R_k:=G_k-D\varphi_k$, so that $F_k={F_*}+D\varphi_k+R_k$. By the construction of $\varphi_k$, we have $R_k\to0$ in $L^p({\omegar};\Rnn)$. 
Recalling that the sequence $P_k-T({F_*})$ is bounded in $L^q({\omegar};\Rnn)$,
{and inserting $F_k-{F_*}=D\varphi_k+R_k$, } Equation (\ref{eqpkdw2fkfab}) {gives}
\begin{equation}
   {\int_{\omegar}} (T({F_*}+D\varphi_k+R_k )- T({F_*}))  \cdot D\varphi_k \,dx\to 0 .
\end{equation}
By (\ref{eqtfgtf}), we obtain
\begin{equation}
    |T({F_*}+D\varphi_k +R_k)- T({F_*}+D\varphi_k)|
    \le c(|{F_*}|^{p-2}+|D\varphi_k|^{p-2}+|R_k|^{p-2}+1)|R_k|    .
\end{equation}
Since  $D\varphi_k$ {is} bounded in $L^p({{\omegar}};\Rnn)$ and $R_k\to0$ in $L^p({{\omegar}};\Rnn)$, with $p\ge  2$, we conclude that
\begin{equation}\label{eqpkdw2fkfb}
    {\int_{\omegar}} (T({F_*}+D\varphi_k)- T({F_*}))  \cdot D\varphi_k \,dx\to 0 ,
\end{equation}
With (\ref{eqdefquasimon}) we obtain
\begin{equation}
\begin{split}
 0=&\limsup_{k\to\infty} \int_{{\omegar}}   (T({F_*}+D\varphi_k)- T({F_*}))  \cdot D\varphi_k \, dx\\
 \ge& \limsup_{k\to\infty} \int_{{\omegar}} B({F_*},D\varphi_k)\, dx .
\end{split}
\end{equation}
{Together with $B\ge 0$} this implies that  the sequence of functions $x\mapsto B({F_*},D\varphi_k(x))$ converges to zero in {$L^1({{\omegar}})$ 
for almost every $r>0$}
and, therefore, in $L^1_\loc(\omega)$. {The proof of (iii) is then concluded as in the previous case, starting right after (\ref{eqbfkfsd}) and recalling that $R_k=F_k-F_*-D\varphi_k$ converges strongly to 0.}

Assertion (iii) and Lemma \ref{lemmalocdivcclsetun} immediately imply (i).

{It remains to prove that (iii) implies (ii).
Assume $(F_k, P_k)\in \Xpq(\omega)$ is \divcurl-convergent to $(F,P)\in\Xpq(\omega)$ with $P_k=T(F_k)$ pointwise almost everywhere. 
Since $(F_k,P_k)$ is bounded in $\Xpq(\omega)$ we can pass to a subsequence 
with $|F_k-F|\weakto h_F$ weakly in  $L^p(\omega)$ and
$|P_k-P|\weakto h_P$ weakly in  $L^q(\omega)$.
We shall show below that
\begin{equation}\label{eqhfhp0}
 h_F=h_P=0 \text{ a.e. in $\omega$}.
\end{equation}
This implies
$\int_\omega|F_k-F|+|P_k-P|dx\to0$. Therefore $(F_k,P_k)\to (F,P)$ strongly in $L^1(\omega;\Rnn\times\Rnn)$ along the chosen subsequence and, by uniqueness of the limit, for the entire sequence, which concludes the proof.
}

{In order to prove (\ref{eqhfhp0}),
we localize using 
the same procedure as in Proposition \ref{propdeltaclosed}, starting from (\ref{eqcurlfkfggk}), with slight modifications to treat also the function $(h_F,h_P)$. In particular, before (\ref{eqdmudcall}) we require $x_*$ to be a Lebesgue point of $(h_F,h_P)$ as well.
In choosing  $\hat k_j$ we require, additionally to (\ref{eqweakconvmed}), also
\begin{equation}\label{eqweakconvmedb}
\begin{split}
    \|R_{r_j}|F_{k}-F|-R_{r_j}h_F\|_{w,p}+
    \|R_{r_j}|P_{k}-P|-R_{r_j}h_P\|_{w,q}&\le
    \frac1j \text{ for all } k\ge \hat k_j.
\end{split}
\end{equation}}
{
Then necessarily
$R_{r_j}(h_F,h_P)\to (h_F,h_P)(x_*)$ and, by the same argument leading to (\ref{eqhatfjpj}), 
\begin{equation}\label{eqfjfh}
\begin{split}
 |\hat F_j-F_*|\weakto h_F(x_*) \text { weakly in } L^p(B_1),\\
 |\hat P_j-P_*|\weakto h_P(x_*) \text { weakly in } L^q(B_1). 
\end{split}
\end{equation}
The rest of the argument, up to  (\ref{propdeltaclosedfin}), is unchanged. 
We conclude that
for almost every $x_*\in\omega$ 
there is a sequence $(\hat F_j,\hat P_j)\in\Xpq(B_1)$
which \divcurl-converges 
to a constant function $(F_*, P_*)$ and which obeys
(\ref{eqfjfh}).}
{By (iii) we obtain $(\hat F_j,\hat P_h)\to (F_*,P_*)$ strongly in $L^1(B_1;\Rnn)$. With (\ref{eqfjfh}) this implies (\ref{eqhfhp0}) and concludes the proof.}
\end{proof}

\subsection{Examples of \divcurl-closed sets in two dimensions}\label{Jlt9at}

We show that the local data sets discussed in Lemma \ref{le:coercive} and Example \ref{example2dcoerciv} are, with appropriate restrictions on the coefficients, \divcurl-closed.

\begin{proposition}\label{propexdeltaclosed}
Let $n=2$. Assume that $g\in C^1(\R)$ is convex, with
\begin{equation}\label{eqgrowthglin}
    |g'(t)|\le {\ceins}+{\czwei}|t|\hskip1cm\text{ for all $t\in\R$}
\end{equation}
for some $b,d\ge0$, and for some $a>0$ set
\begin{equation}
    W_2(\xi):=\frac12 |\xi|^2+\frac14 a|\xi|^4+g(\det \xi).
\end{equation}
If  $b\le 2$ and $d\le 3a$, then $DW_2$ is strictly quasimonotone and the data-set  $\Dloca{DW_2}$ is  locally $(4,\frac43)$-\divcurl-closed.
\end{proposition}

Before proving this result we give an explicit example.

\begin{example}\label{example2dclosed1}
For all $a\in(0,1/4]$ and $\beta\in(0,2a)$, the function
(\ref{eqdefhatw2}) is invariant under rotations, $(4,4/3)$-coercive, minimized by matrices in $\SO(2)$, and generates a $(4,4/3)$-\divcurl-closed data set.
\end{example}

\begin{proof}
Most properties have already been verified in Example \ref{example2dcoerciv}. To check the last one, we note that $g'(t)=\beta t -(\beta+1+2a)$. Since $\beta<2a$ ,the growth condition (\ref{eqgrowthglin}) holds with $d:=\beta<  2a <  3a$ and $b:=\beta+1+2a<1+4a\le 2$. The assertion follows then from Proposition~\ref{propexdeltaclosed}.
\end{proof}

In order to prove Proposition~\ref{propexdeltaclosed}, we show that $DW_2$ is monotone up to null Lagrangians. More precisely we have the following result.

\begin{lemma} \label{le:polymonotone_2d}
Under the assumptions of Proposition~\ref{propexdeltaclosed},
\begin{equation}  \label{eq:polymonotone_2d}
\begin{split}
    &
    (DW_2(F+G) - DW_2(F)) \cdot G
    \\ &
 \hskip10mm   \ge \frac14 a (|G|^2+3F\cdot G)^2+ g'(\det F) \det G
    +(1-\frac b2)|G|^2
\end{split}
\end{equation}
for all $F, G \in \R^{2 \times 2}$.
\end{lemma}

In particular, the function $\hat W_2$ in Example~\ref{example2dclosed1} satisfies \eqref{eq:polymonotone_2d}. We note that $D \hat W_2$ is not monotone. Indeed, if $D\hat W_2$ were monotone, then $\hat W_2$ would be convex. However, the set $\SO(2)$ of minimizers of $\hat W_2$ is not convex. If, additionally, $b<2$, then $W_2$ is polymonotone in the sense of Definition \ref{defpolymonotone} with $A(F)=g'(\det F)$ and $B(F,G)=(1-\frac b2)|G|^2$. The case $b=2$ does not fit directly into the definition of polymonotonicity. However, as shown below, it still results in closedness.

\begin{proof}[{Proof of Lemma \ref{le:polymonotone_2d}}]
Set
\begin{subequations}
\begin{align}
    &
    W_{\text{quadr}}(F) := \frac12|F|^2,
    \\ &
    W_{\text{quart}}(F)
    :=
    \frac14 |F|^4,  \quad
    W_{\text{det}}(F) := g(\det F),
\end{align}
\end{subequations}
so that $W_2=W_{\text{quadr}}+aW_{\text{quart}}+W_{\text{det}}$. We prove below the following inequalities for all $F, G \in \R^{2 \times 2}$,
\begin{subequations}
\begin{align}
    &
    \big( DW_{\text{quadr}}(F+G) -DW_{\text{quadr}}(F) \big)  \cdot G  = |G|^2,
    \label{eq:polymono_quadratic_2d}
    \\
    \begin{split}
        &
        \big( DW_{\text{quart}}(F+G) -DW_{\text{quart}}(F) \big)  \cdot G
        \\\ &\hskip5mm
        \ge \frac34 |F+G|^2 \, |G|^2+\frac14 (|G|^2+3F\cdot G)^2,
    \end{split}
    \label{eq:polymono_quartic_2d}
    \\
    \begin{split}
        &
        \big( DW_{\text{det}}(F+G) -DW_{\text{det}}(F) \big)  \cdot G
         \\ &\hskip5mm
 \ge       g'(\det (F+G)) \det G  + g'(\det F) \det G.
    \end{split}
    \label{eq:polymono_det_2d}
\end{align}
\end{subequations}
Furthermore, from (\ref{eqgrowthglin}) and $|\det\xi|\le \frac12|\xi|^2$ we obtain
\begin{equation}
\begin{split}
    g'(\det (F+G)) \det G &\ge - (d |\det (F +G)| + b) \,  |\det G|
    \\ &\ge
    - \frac{d}{4} |F+G|^2 \, |G|^2 - \frac{b}{2} |G|^2 .
\end{split}
\end{equation}
The four previous estimates combined give, recalling the assumption $d\le 3a$, the claimed inequality \eqref{eq:polymonotone_2d}.

It remains to prove (\ref{eq:polymono_quadratic_2d}-\ref{eq:polymono_det_2d}). The identity    \eqref{eq:polymono_quadratic_2d} is evident. The estimate \eqref{eq:polymono_quartic_2d} is equivalent to
\begin{equation}  \label{eq:quartic_2d_a}
\begin{split}
    &
    |F+G|^2 (F+G) \cdot G - |F|^2 F \cdot G
     \\ &
   \ge \frac34 |F+G|^2 \,
    |G|^2+\frac14 (|G|^2+3F\cdot G)^2 , \quad \forall F, G \in
    \R^{2 \times 2}.
\end{split}
\end{equation}
The estimate holds trivially  for $G=0$. Replacing $F$ by $F/ |G|$ if needed, we may assume that $|G| =1$. We can write $F = \lambda G + F^\perp$ with $\lambda\in\R$ and  $F^\perp \cdot G = 0$. Thus, the left-hand side of \eqref{eq:quartic_2d_a} becomes
\begin{equation}
    \big( (1+ \lambda)^2 + |F^\perp|^2\big)
    ( \lambda + 1) - (\lambda^2  + |F^\perp|^2) \lambda
    =
    (1+\lambda)^3 -\lambda^3 + |F^\perp|^2.
\end{equation}
Hence, \eqref{eq:quartic_2d_a} is equivalent to
\begin{equation}
\begin{split}
    &
    1 + 3 \lambda + 3 \lambda^2 +|F^\perp|^2
    \ge \\ &
    \frac34   (1+
    \lambda)^2 + \frac34 |F^\perp|^2+\frac14 (1+3\lambda)^2 ,\quad \forall \lambda
    \in \R, \, \, F^\perp \in \R^{2 \times 2}.
\end{split}
\end{equation}
To verify this inequality, it suffices to expand both squares. This concludes the proof of (\ref{eq:polymono_quartic_2d}).

Regarding $W_{\text{det}}$, we note that $ DW_{\text{det}}(F+G)) = g'(\det(F+G)) \cof(F+G)$ and
\begin{equation}
    \cof (F+G) \cdot G = \cof F \cdot G + 2 \det G = \det(F+G) - \det F + \det G.
\end{equation}
 Since $g$ is convex, the derivative $g'$ is monotone, i.~e.,
\begin{equation}\label{eqgpmonot}
    g'(t) (t-s) \ge g'(s) (t-s) ,
\end{equation}
for all $s,t \in \R$. Thus,
\begin{eqnarray*}
    & &DW_{\text{det}}(F+G) \cdot G \\&=&  g'(\det (F+G))  \,  (\det (F+G) - \det F) +  g'(\det (F+G)) \det G \\
    & \ge & g'(\det  F)\,  (\det (F+G) - \det F) +  g'(\det (F+G)) \det G ,
\end{eqnarray*}
and, hence,
\begin{eqnarray*}
    & & DW_{\text{det}}(F+G) \cdot G - DW_{\text{det}}(F) \cdot G\\
    & \ge &  g'(\det F) (\det (F+G) - \det F - \cof F \cdot G) +  g'(\det (F+G)) \det G.
\end{eqnarray*}
Since $\det (F+G) =\det F +\cof F \cdot G + \det G$, this concludes the proof of \eqref{eq:polymono_det_2d}.
\end{proof}

\begin{proof}[Proof of Proposition \ref{propexdeltaclosed}]
We first remark that, in the case $b<2$, taking $A(F):=(0,g'(\det  F))$ and $B(F,G):=(1-\frac b2)|G|^2$ one immediately sees that \eqref{eq:polymonotone_2d} implies strict polymonotonicity. The conclusion follows then from Theorem \ref{theostrictlyqmclosed2}.

We now present an {\sl ad hoc} proof {of closedness,} that works for $b\le 2$ {and is based on Lemma \ref{le:polymonotone_2d}}. The first steps are the same as in the proof of  Theorem \ref{theostrictlyqmclosed2}. Fix $(F,P)\in\Rnn\times \Rnn$, $\omega\subseteq\R^n$ open and bounded and consider a sequence $(F_k,P_k)\in X_{4,4/3}(\omega)$ that \divcurl-converges to the constant function $(F,P)$ and such that $P_k=DW_2(F_k)$ a.~e.~for any $k$. We need to show that $P=DW_2(F)$.

By the div-curl Lemma \ref{lemmadivcurl}, we obtain $P_k\cdot F_k\weakto P\cdot F$ and  $\det F_k\weakto \det F$. in the sense of distributions. Weak convergence additionally gives $DW_2(F)\cdot (F_k-F)\weakto0$ and, therefore,
\begin{equation}\label{eqpkdw2fkf}
    (DW_2(F_k) - DW_2(F))  \cdot (F_k - F) \weakto 0 ,
\end{equation}
in the sense of distributions. With $\det (F_k-F)=\det F_k-F_k\cdot \cof F+\det F$ and $F\cdot \cof F=2\det F$, we obtain
\begin{equation}
    \det (F_k -  F) \weakto 0 \hskip1cm \text{ in $L^2(\omega)$}.
\end{equation}
We fix a test function $\theta \in C_c^0(\omega)$ with $\theta \ge 0$. From Lemma \ref{le:polymonotone_2d}, we obtain
\begin{eqnarray*}
    & & \frac{a}{4} \limsup_{k \to \infty}  \int_{\omega}
    (|F_k-F|^2+3F\cdot (F_k-F)
    )^2  \theta \, dx   \\
    & &
    \le    \limsup_{k \to \infty}  \int_{\omega}
    \Big[ \big( DW_2(F_k)  - DW_2(F) \big)  \cdot
    (F_k - F) - g'(\det F) \det (F_k - F) \Big]\, \theta \, dx\\
    & & = 0 .
\end{eqnarray*}
Since $|F_k-F|^2+3F\cdot (F_k-F)=|F_k|^2+F_k\cdot F-2|F|^2$, it follows that
\begin{eqnarray*}
    & & \frac{a}{4} \limsup_{k \to \infty}  \int_{\omega}
    (|F_k|^2+F\cdot F_k-2|F|^2)^2  \theta \, dx   =0.
\end{eqnarray*}
Thus $|F_k|^2+F\cdot F_k-2|F|^2\to0$ in $L^2_\loc(\omega)$. Since $F_k\weakto F$ weakly, we deduce that $|F_k|^2\to |F|^2$, so that $F_k\to F$ strongly in $L^2_\loc(\omega)$. Hence, a subsequence converges a.~e.~and, therefore, $P_k = DW_2(F_k) \to DW_2(F)$ a.~e.~for that subsequence. Thus $P = DW_2(F)$ as desired.
\end{proof}

\subsection{Examples of \divcurl-closed sets in three dimensions}\label{dRl1EW}

 \begin{proposition}\label{propexdeltaclosedthree}
Let $n=3$. There is $c_*>0$ with the following property. Assume that $g\in C^1(\R)$ is convex  and that there is $d\in\R$ such that
\begin{equation}\label{eqgrowthglin3d}
    |g'(t)-g'(s)|\le d(|t|+|s|) \text{ for all } s,t\in\R .
\end{equation}
With $a,e>0$, set
\begin{equation}
    W_3(\xi):=\frac12 |\xi|^2+\frac14 a|\xi|^4+\frac16 e|\xi|^6+g(\det \xi).
\end{equation}
If $d\le c_*e$, then the data-set $\calD_{\loc,DW_3}$ is locally $(6,\frac65)$-\divcurl-closed.
\end{proposition}

\begin{lemma}   \label{pr:quasimonotone_3d}
Let $n=3$ and let $W_3$ be as in Proposition \ref{propexdeltaclosedthree}. There are $c_*>0$ and $c' > 0$  with the following property. If $d \le c_{\ast} e$, then
\begin{equation}
\begin{split}
    \big(  DW_3(F+G) &- DW_3(F) )  \cdot G
    \\
   \ge & c' e  |G|^6 +
    g'(\det F) (F \cdot \cof G + \det G)
     \\ &
+    g'(0) (F \cdot \cof G + 2 \det G) .
\end{split}
\end{equation}
\end{lemma}

\begin{proof}
We have
\begin{equation}  \label{eq:expansion_det_3d}
    \det (F+G) - \det F = \cof F \cdot G + F \cdot \cof G + \det G.
\end{equation}
This implies that
\begin{equation}
    \cof (F +G) \cdot G = \cof F \cdot G + 2 F \cdot \cof G + 3 \det G.
\end{equation}
Indeed, it suffices to apply \eqref{eq:expansion_det_3d} with $sG$ instead of $G$ and differentiate with respect to $s$ at $s=1$
{and use Jacobi's formula $\cof=D\det$.} Thus,
\begin{equation}\label{eqcoffg}
    \cof (F+G) \cdot G = \det(F+G) - \det F + F \cdot \cof G + 2 \det G.
\end{equation}
Set
\begin{equation}
    W_{\text{det}}(F)  := g(\det F).
\end{equation}
Then, {using (\ref{eqcoffg}),}
\begin{align}
    DW_{\text{det}}(F+G) \cdot G = &
    g'(\det(F+G)) \cof(F+G)\cdot G\\
    =&
    g'(\det(F+G))  \,  (\det (F+G) - \det F) +
    R_1 ,
\end{align}
where
\begin{align*}
    R_1 :=& g'(\det (F+G)) (F \cdot \cof G + 2 \det G)
    \\ = &
    g'(0) (F \cdot \cof G + 2 \det G) + R_2
\end{align*}
and $R_2:=(g'(\det (F+G))-g'(0))  (F \cdot \cof G + 2 \det G)$. With $|\cof G|\le |G|^2/\sqrt3$ and $|\det G|\le |G|^3/3^{3/2}$ we obtain
\begin{equation}  \label{eq:error_det_3d}
\begin{split}
    |R_2| &\le  d |\det (F+G)|  \,
    (\frac1{\sqrt3} |F| |G|^2 + \frac2{3^{3/2}} |G|^3)
     \\ &\le
    C_{\ast}  d  |G|^2 (|F|^4 + |G|^4).
\end{split}
\end{equation}
Monotonicity of $g'$ implies, {\sl via} (\ref{eqgpmonot}), that
\begin{equation}
    DW_{\text{det}}(F+G) \cdot G  \ge g'(\det F) (\det (F+G) - \det F) + R_1.
\end{equation}
This yields
\begin{equation}
 \begin{split}
    DW_{\text{det}}(F+G) &\cdot G -  DW_{\text{det}}(F) \cdot G
    \\ \ge &
    g'(\det F) (\det (F+G) - \cof F \cdot G - \det F) + R_1
    \\ = &
    g'(\det F) (F \cdot \cof G + \det G) + R_1.
 \end{split}
\end{equation}
where in the second step we have used (\ref{eq:expansion_det_3d}). Thus,
\begin{equation}
 \begin{split}
 \label{eq:monotonicity_det_3d}
    & DW_{\text{det}}(F+G) \cdot G -  DW_{\text{det}}(F) \cdot G
    \\ &
   \ge  g'(\det F) (F \cdot \cof G + \det G) + g'(0) (F \cdot \cof G + 2 \det G)  + R_2.
 \end{split}
\end{equation}
Let
\begin{equation}
    W_{\text{six}}(F) := \frac16 |F|^6.
\end{equation}
We claim that there exists a $c_{\ast \ast} > 0$ such that
\begin{align} \label{eq:quantitative_monotonicity_six}
    E(F,G)& := \big( DW_{\text{six}}(F+G) - DW_{\text{six}}(F) \big) \cdot G
    \\ &
   \ge  c_{\ast \ast} (|F|^4 + |G|^4) |G|^2 \quad \forall F, G, \in \R^{3 \times 3}.  \nonumber
\end{align}
This inequality is well-known but, for completeness, we include a proof below. Set $c_{\ast} := \frac12 c_{\ast \ast}/ C_{\ast}$ and assume that $d \le c_{\ast} e$. Combining \eqref{eq:monotonicity_det_3d}, \eqref{eq:error_det_3d} and \eqref{eq:quantitative_monotonicity_six} and using the convexity of $F \mapsto |F|^2$ and $F \mapsto |F|^4$, we deduce that
\begin{equation}
 \begin{split}
    \big(  &DW_3(F+G) - DW_3(F) )  \cdot G
     \\ &
    -g'(\det F) (F \cdot \cof G + \det G)   -  g'(0) (F \cdot \cof G + 2 \det G)    \\ & \ge
    \frac12  e c_{\ast \ast} |G|^6 .
 \end{split}
\end{equation}
Hence, the assertion holds with $c' := \tfrac12 c_{\ast \ast}$.

Finally, we recall the proof of \eqref{eq:quantitative_monotonicity_six}. Suppose that the inequality does not hold. Then, for any $k\in\N$ there exist $(F_k, G_k)$ such that
\begin{equation}  \label{eq:Wsix_contradiction}
    E(F_k, G_k) < \frac1k (|F_k|^4 + |G_k|^4) |G_k|^2.
\end{equation}
Both sides are homogeneous of degree six under the rescaling $(F,G) \mapsto (\lambda F, \lambda G)$. Hence, we may assume that $|F_k|^2 + |G_k|^2 = 1$. Passing to a subsequence, we may further assume that $(F_k, G_k) \to (F,G)$ and we deduce that $E(F,G) \le 0$. If $G \ne 0$, this contradicts the strict convexity of $W_{\text{six}}$. If  $G=0$ then $|F_k| \to 1$. Passing to a further subsequence, we may assume that $G_k/  |G_k| \to H$ and $|H| =1$. Dividing \eqref{eq:Wsix_contradiction} by $|G_k|$, we get $D^2 W_{\text{six}}(F)(H,H) \le 0$. But a direct calculation shows that
\begin{equation}
    D^2 W_{\text{six}}(F)(H,H) =  4 |F|^2 (F \cdot H)^2 + |F|^4 |H|^2 \ge 1.
\end{equation}
This contradiction concludes the proof.
\end{proof}

\begin{proof}[Proof of Proposition~\ref{propexdeltaclosedthree}]
Let $c_*>0$ and $c'>0$ be as in Lemma \ref{pr:quasimonotone_3d}. Define  $B(F,G):=c'e|G|^6$ and $A(F):=(0,g'(\det F)F, g'(\det F))$. Then, Lemma \ref{pr:quasimonotone_3d} states
\begin{equation}
    \big(  DW_3(F+G) - DW_3(F) )  \cdot G \ge B(F,G)+A(F)\cdot M(F),
\end{equation}
where $M(F)=(F,\cof F, \det F)$ is the vector of minors of $F$. Therefore, $W_3$ is strictly polymonotone, in the sense of Definition \ref{defpolymonotone}. The conclusion then follows from  Theorem \ref{theostrictlyqmclosed2}.
\end{proof}

\begin{example}\label{example2dclosed}
Let $n=3$, $a,\beta, e>0$, $c_*$ be the minimum between 3 and the constant $c_*$ appearing in Proposition~\ref{propexdeltaclosedthree}.
Assume that $\beta\in (0,c_*e)$ and consider the function $\hat W_3:\R^{3\times 3}\to\R$ defined by
\begin{equation}
    \hat W_3(\xi):=\frac12|\xi|^2+\frac14 a|\xi|^4 + \frac16e|\xi|^6+\frac12 \beta
    (\det \xi-1-\frac{1+3a+9e}{\beta})^2.
\end{equation}
This function is invariant under rotations, $(6,6/5)$-coercive, minimized by matrices in $\SO(3)$, and generates a $(6,6/5)$-\divcurl-closed data set.
\end{example}
\begin{proof}
We first show that the set of minimizers is  $\SO(3)$. Indeed, the arithmetic-geometric mean inequality implies that $|\xi|^2 \ge 3 |\det \xi|^{2/3}$  and equality is achieved if and only if $\xi \in \R \O(3)=\{F\in\R^{3\times 3}: F^TF=t\Id, \text{ for some } t\in\R\}$.  Moreover, $\hat W_3(-\xi) > \hat W_3(\xi)$ if $\det \xi > 0$. Thus,
\begin{equation}
    \min_{\xi \in \R^{3 \times 3}} \hat W_3(\xi) = \min_{t \in [0, \infty)}  f(t) ,
\end{equation}
where
\begin{equation}
    f(t):= \frac32 t^2 + \frac94 a t^4 + \frac92 e t^6 +
    \frac12 \beta (     t^3 - 1 -\frac{1+3a+9e}{\beta})^2.
\end{equation}
We have,
\begin{align*}
    f'(t) = & 3 t^2 + 9 at^4 + 27 et^6 +  (\beta (t^3 - 1) - (1 + 3a + 9e)) 3 t^2
    \\ = &
    9 a(t^4 -t^2) + 27 e(t^6-t^2) + 3 t^2 \beta (t^3-1)
    \\ = &
    3 t^2 \, \big( 3 a(t^2 -1) + 9 e(t^4 - 1) +  \beta (t^3-1) \big).
 \end{align*}
Thus, $f'(t) < 0$ for $t \in [0,1)$ and $f'(t) > 0$ in $(1, \infty]$. This implies that $f$ has a strict global minimum at $1$. Hence, the set of minimizers of $W_3$ is $\SO(3)$.

In order to check closedness and coercivity, we define  $g(t):=\frac12 \beta(t-1-\frac{1+3a+9e}\beta)^2$ and compute $g'(t)=\beta t -(\beta+1+3a+9e)$. Since $\beta<3e$, Lemma \ref{le:coercive}(ii) proves coercivity. From $|g'(t)-g'(s)|\le \beta|t-s|$ and $\beta<c_*e$, we see that Proposition~\ref{propexdeltaclosedthree} applies and proves closedness.
\end{proof}

\section{Existence of minimizers}\label{6hoXUS}

On the strength of the preceding developments, we are finally in a position to elucidate the question of existence of solutions.
Throughout this section we work on local materials, in the sense of Definition \ref{deflocalmat}, and  assume that Assumptions \ref{assumptiongeneral} and \ref{Ze6How} hold. We write, as in the Introduction,
\begin{equation}
    J(F,P)
    :=
        \begin{cases}
            \displaystyle\int_\Omega  \psi_{\calD_\loc}(F(x),P(x)) \, dx ,
            & \text{if } (F,P) \in \Emb , \\
            \infty , & \text{otherwise} ,
        \end{cases}
\end{equation}
whereupon problem (\ref{9lFapr}) becomes
\begin{equation}
    \inf_{(F,P) \in \Xpq(\Omega)} J(F,P) .
\end{equation}
We also write
\begin{equation}
    I((F,P),(F',P'))
    :=
    \begin{cases}
\displaystyle    \int_\Omega \Big( V(F(x)-F'(x))+V^*(P(x)-P'(x)) \Big) \, dx ,\\
\hfill \text{ if } (F,P)\in\Emb,\ (F',P')\in \calD,
     \\
     \infty, \hfill \text{ otherwise.}
    \end{cases}
\end{equation}
whereupon problem (\ref{6Hudof}) becomes
\begin{equation}\label{6Hudof2}
    \inf_{((F,P), (F',P'))\in \Xpq(\Omega)\times \Xpq(\Omega)}
    I((F,P),(F',P')) .
\end{equation}
From the definition it is also clear that
\begin{equation}
    J(F,P)=
    \inf_{(F',P')\in \Xpq(\Omega)}
    I((F,P),(F',P')).
\end{equation}
The set $\Emb$ is not empty by definition. If $\calD$ is not empty, then $\inf I<\infty$ and $\inf J<\infty$.

A number of alternatives arise in connection with the possible solutions of problem (\ref{6Hudof2}).

\begin{definition}[Classical solution]
Given a stress-strain function $T:\Rnn\to\Rnn$, we say that $u\in W^{1,p}(\Omega;\R^n)$ is a classical solution if $(Du, T(Du))\in \calE$.
\end{definition}

This definition implies the satisfaction of the boundary conditions, compatibility and the equilibrium equation as in Definition~\ref{defineE}. Clearly, classical solutions are minimizers of (\ref{6Hudof2}). In particular, classical solutions exist if and only if the infimum (\ref{6Hudof2}) is attained and is zero. A similar notion of solution can be defined without recourse to the assumption that the data set is local and is generated by a stress-strain function.

\begin{definition}[Strong solution]
Given a data set $\calD\subseteq\Xpq(\Omega)$, we say that $(F,P)\in \Xpq(\Omega)$ is a strong solution if $(F,P)\in \calE\cap\calD$.
\end{definition}

In general we cannot expect the infimum of $I$ to be zero, which suggests the following generalization of the notion of solution.

\begin{definition}[Generalized solution]
We say that $((F,P),(F',P')) \in \calE\times\calD\subseteq\Xpq(\Omega) \times \Xpq(\Omega)$ is a generalized solution if it is a minimizer of (\ref{6Hudof2}) in the set $\calE\times\calD$.
\end{definition}

Any classical solution $u\in W^{1,p}(\Omega;\R^n)$ corresponds to a strong solution given by $F:=Du$, $P:=T(F)$. Furthermore, all strong  solutions are generalized solutions and a generalized solution is strong if and only if $F'=F$ and $P'=P$. We make these assertions precise in the following theorem.

\begin{theorem}\label{theosoleq}
{The} following are equivalent:
\begin{enumerate}
 \item \label{theosoleqitstrongs} $(F,P)\in \Xpq(\Omega)$ is a strong solution.
 \item \label{theosoleqitjzero}  $J(F,P)=0$.
 \item \label{theosoleqitgens}  $((F,P), (F,P))\in \Xpq(\Omega)\times \Xpq(\Omega)$ is a generalized solution.
\end{enumerate}
If $\calD$ is generated by a stress-strain function $T:\Rnn\to\Rnn$, then the following is also equivalent:
\begin{enumerate}
  \setcounter{enumi}{3}
 \item\label{theosoleqitclass} there is a classical solution $u\in W^{1,p}(\Omega;\R^n)$ with $F=Du$, $P=T(Du)$.
\end{enumerate}
 \end{theorem}
The proof is straightforward from the definitions. For the convenience of the reader, we provide some details.

 \begin{proof}
\ref{theosoleqitstrongs} $\Rightarrow$ \ref{theosoleqitjzero}.
Let $(F,P)\in \calD\cap\calE$. Then, $(F(x),P(x))\in\calD_\loc$ almost everywhere, hence, $\psi_{\calD_\loc}(F,P)=0$ almost everywhere.

\ref{theosoleqitjzero} $\Rightarrow$ \ref{theosoleqitstrongs}.
If $J(F,P)=0$ then necessarily $(F,P)\in \calE$. Since $\psi_{\calD_\loc}>0$ away from $\calD_\loc$, we also have $(F(x),P(x))\in\calD_\loc$ almost everywhere and, therefore, $(F,P)\in\calD$.

\ref{theosoleqitstrongs}
$\Rightarrow$ \ref{theosoleqitgens}.
Follows from $(F,P)\in\calD\cap\calE$, $I((F,P),(F,P))=0$, and $I\ge0$ for all arguments.

\ref{theosoleqitgens} $\Rightarrow$ \ref{theosoleqitstrongs}.
By definition of generalized solution, $(F,P)\in\calE\times\calD$.

We now assume now that $\calD_\loc=\DlocT$ for some  $T:\Rnn\to\Rnn$.

\ref{theosoleqitstrongs} $\Rightarrow$ \ref{theosoleqitclass}.
Since $(F,P)\in \calE$, there is $u\in W^{1,p}(\Omega;\R^n)$ with $F=Du$. Since $(F,P)(x)\in \calD_\loc$, we have $P=T(F)=T(Du)$ almost everywhere,, therefore, $(Du, T(Du))\in\calE$.

\ref{theosoleqitclass} $\Rightarrow$ \ref{theosoleqitstrongs}.
We set $F:=Du$, $P:=T(Du)$, so that automatically $(F,P)\in\calD$. Since $u$ is a classical solution, $(F,P)\in \calE$.
\end{proof}

We now show how the concepts of coercivity and closedness developed in the previous sections can be used to prove existence of solutions, provided that the infimum of $J$ is known to be zero.

\begin{theorem}\label{prA5lj}
Let $\calD_\loc$ be nonempty, $(p,q)$-coercive and locally \divcurl-closed. If $\inf J=0$, there is a minimizer of $J$ in $\calE$ that is a strong solution.
\end{theorem}

\begin{proof}
Let $(F_k,P_k)$ be a minimizing sequence in $\calE$. By Theorem \ref{theoremcompactness}, there is a subsequence that \divcurl-converges to some $(F,P)\in\calE$. By \divcurl-closedness $(F,P)\in\calD$,, therefore, it is a minimizer of $J$ in $\calE$.
\end{proof}

The closedness property can be verified {\sl via} Theorem \ref{theostrictlyqmclosed2} and holds, in particular, for the examples presented in the previous sections. The property $\inf J=0$ depends on the boundary values and remains to be verified explicitly on a case-by-case basis.

\begin{remark}
If $\calD_\loc=\DlocT$ for some $T:\R^{n\times n}\to\R^{n\times n}$ that is $(p,q)$-coercive, in the sense of Remark \ref{remTcoercive}, and one of the following conditions holds:
\begin{enumerate}
    \item[(a)] $T$ is continuous and strictly polymonotone;
    \item[(b)]$T$ is strictly quasimonotone and $DT$ obeys the growth condition (\ref{eqtfgtf}).
\end{enumerate}
Then, $\calD_\loc$ is nonempty, $(p,q)$-coercive and locally \divcurl-closed, as required in Theorem \ref{prA5lj}.
\end{remark}

\begin{proof}
The proof follows immediately from Theorem \ref{theostrictlyqmclosed2}.
\end{proof}

\appendix
\section{Traces of Sobolev spaces}
\label{appendixtraces}
We use standard properties of Sobolev spaces and their traces, see for example \cite{KufnerJohnFucik1977, Temam1979BuchNavierStokes, Ziemer1989-weaklydiff-book}. For the convenience of the reader we recall here the basic definitions and the facts used in the preceding analyses. 

\begin{definition}\label{deffractionalsobolev} Let $\Omega\subseteq\R^n$ be a bounded Lipschitz set, $p\in(1,\infty)$. For $f\in L^p(\partial\Omega;\R^n)$ we define
\begin{equation}
    [f]_{1-1/p,p}^p:=
    \int_{\partial\Omega\times \partial\Omega}
    \frac{|f(x)-f(y)|^p}{|x-y|^{n+1}} d\calH^{n-1}(x) d\calH^{n-1}(y) ,
\end{equation}
and set
\begin{equation}
    W^{1-1/p,p}(\partial\Omega;\R^n):=\{f\in  L^p(\partial\Omega;\R^n):
    [f]_{1-1/p,p}^p<\infty\},
\end{equation}
equipped with the norm $\|f\|_{L^p(\partial\Omega)}+[f]_{1-1/p,p}$. Furthermore, the dual space is denoted by
\begin{equation}
    W^{-1+1/p,q}(\partial\Omega;\R^n)
    =
    (W^{1-1/p,p}(\partial\Omega;\R^n))^*,
\end{equation}
where $q>1$ is defined by the condition $1/p+1/q=1$.
\end{definition}
It is readily checked that $W^{1-1/p,p}(\partial\Omega;\R^n)$ is a reflexive Banach space
\cite[Sect.~6.8]{KufnerJohnFucik1977}.

\begin{lemma}\label{lemmatracew1p}
Let $\Omega\subseteq\R^n$ be a bounded Lipschitz set, $p\in (1,\infty)$.
\begin{enumerate}
\item There is a linear continuous operator $B:W^{1,p}(\Omega;\R^n)\to W^{1-1/p,p}(\partial\Omega;\R^n)$ such that $B\varphi=\varphi|_{\partial\Omega}$ for any $\varphi\in C^1(\bar\Omega;\R^n)$.
\item There is a linear continuous operator $\Ext: W^{1-1/p,p}(\partial\Omega;\R^n)\to W^{1,p}(\Omega;\R^n)$ such that $B\Ext u=u$ for any $u\in  W^{1-1/p,p}(\partial\Omega;\R^n)$.
\end{enumerate}
\end{lemma}
\begin{proof}
See \cite[Th.~6.8.13 and Th.~6.9.2]{KufnerJohnFucik1977}
\end{proof}

\begin{definition}\label{defEqomega}
Let $\Omega\subseteq\R^n$ be a bounded Lipschitz set, $q\in (1,\infty)$. We define
\begin{equation}
    E^q(\Omega):=\{ v\in L^q(\Omega;\Rnn): \div v\in L^q(\Omega;\R^n)\} ,
\end{equation}
and endow it with the norm $\|v\|_{E^q}:=\|v\|_{L^q}+\|\div v\|_{L^q}$.
\end{definition}
It is easily verified that $E^q$ is a reflexive Banach space, and that $C^\infty(\bar\Omega)\cap E^q(\Omega)$ is dense in $v$, see, for example, \cite[Th.~1.1]{Temam1979BuchNavierStokes} (the different exponent makes no difference in the proof).
\begin{lemma}\label{lemmatraceeq}
Let $\Omega\subseteq\R^n$ be a bounded Lipschitz set, $q\in (1,\infty)$. There is a linear continuous operator $B_\nu:E^q(\Omega)\to W^{-1/q,q}(\partial\Omega;\R^n)$ such that $B_\nu\varphi=\varphi|_{\partial\Omega}\cdot \nu$, where $\nu$ is the outer normal to $\partial\Omega$, for any $\varphi\in C^1(\bar\Omega;\R^n)\cap E^q(\Omega)$. Furthermore,
\begin{equation}\label{eqdualitybnub}
    \langle B_\nu v, Bu\rangle = \int_\Omega v\cdot Du \,dx + \int_\Omega u \cdot \div v \,dx ,
\end{equation}
for all $v\in E^q(\Omega)$, $u\in W^{1,p}(\Omega;\R^n)$, where $\langle \cdot,\cdot\rangle$ denotes the duality pairing between $ W^{1/q,p}(\partial\Omega;\R^n)$ and $ W^{-1/q,q}(\partial\Omega;\R^n)$.
\end{lemma}
\begin{proof}
Follows from the same argument used in \cite[Th.~1.2 and Rem.~1.3]{Temam1979BuchNavierStokes}
\end{proof}

\section*{Acknowledgements}

This work was partially funded by the Deutsche Forschungsgemeinschaft (DFG, German Research Foundation) {\sl via} project 211504053 - SFB 1060 and project 390685813 -  GZ 2047/1 - HCM.


\end{document}